%% file: construction-bimodule-categories-hyperfinite-case.tex
\documentclass[leqno,11pt]{article}

\usepackage[T1]{fontenc}

\usepackage{amsmath,amssymb,amsthm}
\usepackage[all]{xy}
\usepackage{mathabx}

\numberwithin{equation}{section}

\newtheorem{theoremcounter}{theoremcounter}[section]
\newtheorem{thmstarcounter}{thmstarcounter}

\newtheorem{corollary}[theoremcounter]{Corollary}
\newtheorem{lemma}[theoremcounter]{Lemma}

\newtheorem{proposition}[theoremcounter]{Proposition}
\newtheorem{theorem}[theoremcounter]{Theorem}
\newtheorem{thmstar}[thmstarcounter]{Theorem}

\theoremstyle{definition}

\newtheorem{definition}[theoremcounter]{Definition}

\newtheorem{remark}[theoremcounter]{Remark}

\input{shortcuts}

\begin{document}

\begin{center}
{\LARGE \bf {Tensor \Cstar-categories arising as bimodule categories\\ of II$_1$ factors}}

\bigskip

{\sc by S\'{e}bastien Falgui\`{e}res$^{(1)}$ and Sven Raum$^{(2)}$
\setcounter{footnote}{1}
\footnotetext{Partially supported by the Fondation des Sciences Math\'{e}matiques de Paris}
    \setcounter{footnote}{2}\footnotetext{Partially supported by Marie Curie Research Training Network Non-Commutative Geometry MRTN-CT-2006-031962 and by K.U.Leuven BOF research grant OT/08/032}
    }
\end{center}

\begin{abstract}
\noindent We prove that if $\cC$ is a tensor \Cstar-category in a certain class, then there exists an uncountable family of pairwise non stably isomorphic II$_1$ factors $(M_i)$ such that the bimodule category of $M_i$ is equivalent to $\cC$ for all $i$. In particular, we prove that every finite tensor \Cstar-category is the bimodule category of a II$_1$ factor. As an application we prove the existence of a II$_1$ factor for which the set of indices of finite index irreducible subfactors is $\left \{1, \frac{5 + \sqrt{13}}{2}, 12 + 3\sqrt{13}, 4 + \sqrt{13}, \frac{11 + 3\sqrt{13}}{2}, \frac{13 + 3\sqrt{13}}{2}, \frac{19 + 5\sqrt{13}}{2}, \frac{7 + \sqrt{13}}{2} \right \}$. We also give the first example of a II$_1$ factor $M$ such that $\bimod(M)$ is explicitly calculated and has an uncountable number of isomorphism classes of irreducible objects.
\end{abstract}

\input{introduction}
\input{acknowledgements}
\input{preliminaries}
\input{main-part4}
\input{applications}
\input{appendix}

{\small\parbox[t]{250pt}{S\'{e}bastien Falgui\`{e}res \\ Laboratoire de Math\'{e}matiques Nicolas Oresme\\ Universit\'{e} de Caen-Basse-Normandie \\ 14032 Caen cedex \\ France\\
{\footnotesize sebastien.falguieres@unicaen.fr}}
\hspace{15pt}
\parbox[t]{150pt}{Sven Raum\\ Department of Mathematics\\
   K.U.Leuven, Celestijnenlaan 200B\\ B--3001 Leuven \\ Belgium
   \\ {\footnotesize sven.raum@wis.kuleuven.be}}}

\bibliographystyle{mybibtexstyle}
\bibliography{operatoralgebras}

\end{document}

%% file: shortcuts.tex
\newcommand{\cA}{\ensuremath{\mathcal{A}}}

\newcommand{\cC}{\ensuremath{\mathcal{C}}}
\newcommand{\cD}{\ensuremath{\mathcal{D}}}

\newcommand{\cF}{\ensuremath{\mathcal{F}}}
\newcommand{\cG}{\ensuremath{\mathcal{G}}}
\newcommand{\cH}{\ensuremath{\mathcal{H}}}
\newcommand{\cI}{\ensuremath{\mathcal{I}}}

\newcommand{\cK}{\ensuremath{\mathcal{K}}}
\newcommand{\cL}{\ensuremath{\mathcal{L}}}

\newcommand{\cP}{\ensuremath{\mathcal{P}}}

\newcommand{\cU}{\ensuremath{\mathcal{U}}}

\newcommand{\cX}{\ensuremath{\mathcal{X}}}
\newcommand{\cY}{\ensuremath{\mathcal{Y}}}
\newcommand{\cZ}{\ensuremath{\mathcal{Z}}}

\newcommand{\rB}{\ensuremath{\mathrm{B}}}

\newcommand{\rE}{\ensuremath{\mathrm{E}}}

\newcommand{\rL}{\ensuremath{\mathrm{L}}}
\newcommand{\rM}{\ensuremath{\mathrm{M}}}

\newcommand{\rZ}{\ensuremath{\mathrm{Z}}}

\newcommand{\ol}{\overline}
\newcommand{\wt}{\widetilde}
\newcommand{\wh}{\widehat}

\newcommand{\amid}{\ensuremath{\, | \,}}

\newcommand{\eqstop}{\ensuremath{\, \text{.}}}
\newcommand{\eqcomma}{\ensuremath{\, \text{,}}}

\newcommand{\NN}{\ensuremath{\mathbb{N}}}
\newcommand{\ZZ}{\ensuremath{\mathbb{Z}}}
\newcommand{\QQ}{\ensuremath{\mathbb{Q}}}
\newcommand{\RR}{\ensuremath{\mathbb{R}}}
\newcommand{\CC}{\ensuremath{\mathbb{C}}}

\newcommand{\Hom}{\ensuremath{\mathop{\mathrm{Hom}}}}
\newcommand{\id}{\ensuremath{\mathrm{id}}}
\newcommand{\ra}{\ensuremath{\rightarrow}}

\newcommand{\Cmat}[1]{\ensuremath{\mathrm{M}_{#1}(\CC)}}
\newcommand{\Cmatnm}[2]{\ensuremath{\mathrm{M}_{#1, #2}(\CC)}}

\newcommand{\tr}{\ensuremath{\mathop{\mathrm{tr}}}}
\newcommand{\Tr}{\ensuremath{\mathop{\mathrm{Tr}}}}

\newcommand{\Cstar}{\ensuremath{\text{C}^*}}

\newcommand{\twonorm}{\ensuremath{\| \cdot \|_2}}
\newcommand{\supp}{\ensuremath{\mathop{\mathrm{supp}}}}
\newcommand{\bimod}{\ensuremath{\mathop{\mathrm{Bimod}}}}
\newcommand{\vNtensor}{\ensuremath{\overline{\otimes}}}
\newcommand{\bim}[3]{\mathord{\raisebox{-0.4ex}[0ex][0ex]{\scriptsize $#1$}{#2}\hspace{-0.2ex}\raisebox{-0.4ex}[0ex][0ex]{\scriptsize $#3$}}}
\newcommand{\lmo}[2]{\mathord{\raisebox{-0.4ex}[0ex][0ex]{\scriptsize $#1$}{#2}}}
\newcommand{\Ltwo}{\ensuremath{\mathrm{L}^2}}
\newcommand{\ltwo}{\ensuremath{\ell^2}}
\newcommand{\cspan}{\ensuremath{\mathop{\overline{\mathrm{span}}}}}
\newcommand{\dimvec}{\ensuremath{\mathop{\mathrm{dim}}}}

\newcommand{\Aut}{\ensuremath{\mathrm{Aut}}}
\newcommand{\Ad}{\ensuremath{\mathop{\mathrm{Ad}}}}
\newcommand{\Out}{\ensuremath{\mathrm{Out}}}
\newcommand{\Rep}{\ensuremath{\mathrm{Rep}}}
\newcommand{\grpaction}[1]{\ensuremath{\stackrel{#1}{\curvearrowright}}}

\newcommand{\SL}{\operatorname{SL}}
\newcommand{\FAlg}{\operatorname{Falg}}
\newcommand{\ANFAlg}{\operatorname{AFalg}}
\newcommand{\QN}{\operatorname{QN}}
\newcommand{\B}{\operatorname{B}}

\newcommand{\M}{\operatorname{M}}
\newcommand{\grp}{\operatorname{grp}}

\newcommand{\norm}{\operatorname{Norm}}

\newcommand{\Repfin}{\operatorname{Rep_{fin}}}

\newcommand{\Comm}{\operatorname{Comm}}
\newcommand{\eL}{\operatorname{L}}
\newcommand{\Bimod}{\operatorname{Bimod}}

\newcommand{\N}{\mathbb{N}}
\newcommand{\Z}{\mathbb{Z}}
\newcommand{\R}{\mathbb{R}}
\newcommand{\C}{\mathbb{C}}
\newcommand{\Q}{\mathbb{Q}}

\newcommand{\al}{\alpha}

\newcommand{\Om}{\Omega}

\newcommand{\vfi}{\varphi}

\newcommand{\actson}{\curvearrowright}

\newcommand{\droite}{\rightarrow}

\newcommand{\ot}{\otimes}

\newcommand{\ox}{\overline{x}}

\newcommand{\psg}{\langle}
\newcommand{\psd}{\rangle}

\newcommand{\vnt}{\vNtensor}

\newcommand{\CoRep}{\mathrm{CoRep}}
\newcommand{\CoRepfin}{\CoRep_{\mathrm{fin}}}
\newcommand{\linfty}{\ell^\infty}

%% file: introduction.tex
\section{Introduction}

The description of symmetries of a II$_1$ factor $M$, such as the \emph{fundamental group} $\cF(M)$ of Murray and von Neumann and the \emph{outer automorphism group} $\Out(M)$, is a central and usually very hard problem in the theory of II$_1$ factors. Over the last ten years, Sorin Popa developed his \emph{deformation-rigidity} theory \cite{popa06-betti-numbers, popa06-strong-rigidity-1, popa06-strong-rigidity-2} and settled many long standing open problems in this direction. See \cite{vaes07, popa07-deformation-rigidity, vaes11} for a survey. In particular, he obtained the first complete calculations of fundamental groups \cite{popa06-betti-numbers} and outer automorphism groups \cite{ioanapetersonpopa08}. His methods were used in further calculations. Without being exhaustive, see for example \cite{popa06-strong-rigidity-1, popavaes10, deprez10} concerning fundamental groups and \cite{popavaes08, vaes08, falguieresvaes08} for outer automorphism groups.

Bimodules $\bim{M}{\cH}{M}$ over a II$_1$ factor $M$ having finite left and right $M$-dimension are said to be of finite \emph{Jones index} (see \cite{connes94, popa86}) and they give rise to a category, which we denote by $\Bimod(M)$.  Endowed with the Connes tensor product of $M$-$M$-bimodules, $\Bimod(M)$ is a compact tensor C$^*$-category, in the sense of Longo and Roberts \cite{longoroberts97}.

The bimodule category of a II$_1$ factor $M$ may be seen as a \emph{generalized symmetry group} of $M$. It contains a lot of structural information on $M$ and encodes several other invariants of $M$. Indeed, if $\grp(M)$ denotes the group-like elements in $\Bimod(M)$, i.e. bimodules of index $1$, one has the following short exact sequence
\[ 1 \to \Out(M) \to \grp(M) \to \cF(M) \to 1 \, .\]
Finite index subfactors $N \subset M$ are also encoded in a certain sense by the bimodule category $\Bimod(M)$, since, denoting $N \subset M \subset M_1$ the \emph{Jones basic construction}, we obtain a finite index bimodule $\bim{M}{\rL^2(M_1)}{M}$.

As explained above, in \cite{ioanapetersonpopa08} the first actual computation of the outer automorphism group of II$_1$ factors was achieved, using a combination of relative property (T) and amalgamated free products. Extending their methods, in \cite{vaes09}, Vaes proved the existence of a II$_1$ factor $M$ with trivial bimodule category. As a consequence, all the symmetry groups and subfactors of $M$ were trivial. Also relying on Popa's methods, in \cite{falguieresvaes10}, Vaes and the first author proved that the representation category of any compact second countable group can be realized as the bimodule category of a II$_1$ factor. More precisely, given a compact second countable group $G$, there exists a II$_1$ factor $M$ and a minimal action $G \actson M$ such that, denoting $M^G$ the fixed point II$_1$ factor, the natural fully faithful tensor functor $\Rep(G) \to \Bimod(M^G)$ is an equivalence of tensor $\Cstar$-categories. Both papers followed closely \cite{ioanapetersonpopa08} and thus, they give only existence results.

Explicit results on the calculation of bimodule categories are obtained in \cite{vaes08} and \cite{deprezvaes10}. Both articles are based on generalizations of Popa's seminal papers \cite{popa06-betti-numbers, popa06-strong-rigidity-1} on Bernoulli crossed products. In \cite{vaes08}, Vaes gave explicit examples of group-measure space II$_1$ factors $M$ for which the fusion algebra, i.e. isomorphism classes of finite index bimodules and fusion rules, were calculated. The complete calculation of the category of bimodules over II$_1$ factors coming from \cite{vaes08} was obtained by Deprez and Vaes in \cite{deprezvaes10}. Even more is proven in \cite{deprezvaes10}, since the C$^*$-bicategory of II$_1$ factors commensurable with $M$, i.e. those II$_1$ factors $N$ admitting a finite index $N$-$M$-bimodule, is also computed and explicitly arises as the bicategory associated with a Hecke pair of groups.

Note that by \cite{doplicherroberts89}, representation categories of compact second countable groups can be characterized abstractly as symmetric compact tensor \Cstar-categories with countably many isomorphism classes of irreducible objects. Among compact tensor \Cstar-categories, \emph{finite tensor \Cstar-categories}, i.e. those which admit only finitely many isomorphism classes of irreducible objects, form another natural class. In this article, we prove that every finite tensor \Cstar-category arises as the bimodule category of a II$_1$ factor.

\begin{thmstar}
\label{thm:finite-tensor-categories}
Let $\cC$ be a finite tensor \Cstar-category. Then there is a II$_1$ factor $M$ such that $\bimod(M) \simeq \cC$.
\end{thmstar}

As an application of the above theorem, we prove the existence of a II$_1$ factor for which the set of indices of irreducible finite index subfactors can be explicitly calculated and contains irrationals. Recall the amazing theorem of Jones, proving that the index of an inclusion of II$_1$ factors $N \subset M$ ranges in the set
\[\cI = \left \{4\cos \left (\frac{\pi}{n} \right )^2 \mid  n = 3, 4, 5, \dotsc \right \} \cup [4, +\infty] \eqstop \]
Given a II$_1$ factor $M$, Jones defines the invariant
\[ \cC(M) = \left \{ \lambda \mid \text{there is a finite index irreducible inclusion } N \subset M \text{ of index}\ \lambda \right \} \eqstop\]
Jones proved that every element of $\cI$ arises as the index of a not necessarily irreducible subfactor of the hyperfinite II$_1$ factor. However, the problem of computing $\cC(R)$ is still widely open. In \cite{vaes08, vaes09}, Vaes proved the existence of II$_1$ factors $M$ for which $\cC(M) = \{1\}$ and $\cC(M) = \{ n^2 \mid n \in \N\}$. The invariant $\cC(M)$ is also computed in \cite{falguieresvaes10} and arises as the set of dimensions of some finite dimensional von Neumann algebras.  In \cite{deprezvaes10}, Deprez and Vaes constructed concrete group-measure space II$_1$ factors $M$ with $\cC(M)$ ranging over all sets of natural numbers that are closed under taking divisors and taking lowest common multiples.

All above results provide II$_1$ factors $M$ for which $\cC(M)$ is a subset of the natural numbers. However,  combining recent work on tensor categories \cite{grossmansnyder11} and our Theorem \ref{thm:finite-tensor-categories}, we prove the following theorem.

\begin{thmstar}
\label{thm:index-sets}
There exists a II$_1$ factor $M$ such that
\[\cC(M) =
  \left \{1, \frac{5 + \sqrt{13}}{2}, 12 + 3\sqrt{13}, 4 + \sqrt{13},
         \frac{11 + 3\sqrt{13}}{2}, \frac{13 + 3\sqrt  {13}}{2},
         \frac{19 + 5\sqrt{13}}{2}, \frac{7 + \sqrt{13}}{2}  \right \} \eqstop\]
\end{thmstar}

In \cite{vaes08}, \cite{falguieresvaes10} and \cite{deprezvaes10} only categories with at most countably many isomorphism classes of irreducible objects were obtained as bimodule categories of II$_1$ factors. In this article we give examples of II$_1$ factors $M$ such that $\bimod(M)$ can be calculated and has uncountably many pairwise non isomorphic irreducible objects. For example, if $G$ is a countable, discrete group, we prove the existence of a II$_1$ factor $M$ such that $\bimod(M) \simeq \Repfin(G)$. Here, $\Repfin(G)$ denotes the category of finite dimensional unitary representations of $G$.

\begin{thmstar}
\label{thm:representation-categories}
Let $\cC$ denote one of the following compact tensor $\Cstar$-categories. Either $\cC = \Repfin(G)$ for a countable discrete group, or $\cC = \CoRepfin(A)$ for an amenable or a maximally almost periodic discrete Kac algebra $A$. Then, there is a II$_1$ factor $M$ such that $\bimod(M) \simeq \cC$.
\end{thmstar}

Our construction consists of two main steps.
\begin{enumerate}
\item
Given any quasi-regular, depth 2 inclusion $N \subset Q$ of II$_1$ factors, such that $N$ and $N' \cap Q$ are hyperfinite, denote by $N \subset Q \subset Q_1$ the Jones basic construction.  We construct a II$_1$ factor $M$ and a fully faithful tensor \Cstar-functor $F: \bimod(Q \subset Q_1) \ra \bimod(M)$ (see Section \ref{sec.Bimod_subfactor} for the bimodule category associated with an inclusion of II$_1$ factors).
\item
Using Ioana, Peterson and Popa's rigidity results for amalgamated free product von Neumann algebras \cite{ioanapetersonpopa08}, we prove that under suitable assumptions (see Theorem \ref{thm.main}) the functor $F$ is essentially surjective.
\end{enumerate}

The above steps yield a II$_1$ factor $M$ such that $\bimod(M) \simeq \bimod(Q \subset Q_1)$. Using the setting of \cite{ioanapetersonpopa08}, as in \cite{falguieresvaes08, falguieresvaes10, vaes09}, this result is not constructive. We only prove an existence theorem, which involves a Baire category argument (see Theorem \ref{theo.free}). More precisely, we prove the following Theorem \ref{thm:construct-bimodule-categories} and Theorems \ref{thm:finite-tensor-categories} and \ref{thm:representation-categories} are obtained as corollaries.

\begin{thmstar}
\label{thm:construct-bimodule-categories}
Let $N \subset Q$ be a quasi-regular and depth 2 inclusion of II$_1$ factors. Assume that $N$ and $N' \cap Q$ are hyperfinite and denote by $N \subset Q \subset Q_1$ the basic construction. Then, there exist uncountably many pairwise non-stably isomorphic II$_1$ factors $(M_i)$ such that for all $i$ we have $\bimod(M_i) \simeq \bimod(Q \subset Q_1)$ as tensor \Cstar-categories.
\end{thmstar}

%%% Local Variables: 
%%% mode: latex
%%% TeX-master: "construction-bimodule-categories-hyperfinite-case"
%%% End: 

%% file: acknowledgements.tex
\subsubsection*{Acknowledgements}

The second author wants to thank his advisor Stefaan Vaes for his constant care and excellent supervision. This work was strongly influenced by his ideas. The second author is greatful to the Laboratoire de Math\'{e}matiques Nicolas Oresme at the university of Caen for its kind hospitality during his stay in autumn 2010. Both authors want to thank the Institut Henri Poincar\'{e} where part of this article has been written, during the program ``Von Neumann algebras and ergodic theory of group action''  in spring 2011.

%% file: preliminaries.tex
\section{Preliminaries and notations}
In this paper, von Neumann algebras are assumed to act on a separable Hilbert space. A von Neumann algebra $(M,\tau)$ endowed with a faithful normal tracial state $\tau$ is called a tracial von Neumann algebra. We define $\Ltwo(M)$ as the GNS Hilbert space with respect to $\tau$.

Whenever $M$ is a von Neumann algebra, we write $M^n = \M_n(\CC) \ot M$ and $M^\infty = \B(\ltwo(\N)) \vnt M$. Whenever $\cH$ is a Hilbert space, we also denote $\cH^\infty = \ltwo(\NN) \ot \cH$.

 If $B \subset M$ is a tracial inclusion of von Neumann algebras, then we denote by $\rE_B$ the trace preserving conditional expectation of $M$ onto $B$. Also if $pB^np \subset pM^np$ is an amplification of $B \subset M$, we still denote by $\rE_B$ the trace preserving conditional expectation onto $pB^np$.

\subsection{Finite index bimodules}\label{section.fibimodules}
Let $M$, $N$ be tracial von Neumann algebras. An $M$-$N$-bimodule $\bim{M}{\cH}{N}$ is a Hilbert space $\cH$ equipped with a normal representation of $M$ and a normal anti-representation of $N$ that commute. Bimodules over von Neumann algebras were studied in \cite[V.Appendix B]{connes94} and \cite{popa86}.

Let $\cH$ be an $M$-$N$-bimodule. There exists a projection $p \in N^\infty$ such that $$\cH_N \cong \bigl( p \eL^2(N)^\infty \bigr)_N \; \; ,$$ and this projection $p$ is uniquely defined up to equivalence of projections in $N^\infty$. There also exists a $*$-homomorphism $\psi : M \droite p N^\infty p$ such that $\bim{M}{\cH}{N}$ is isomorphic with the $M$-$N$-bimodule $\cH(\psi)$ defined as Hilbert space $p\rL^2(N)^\infty$ and endowed with actions given by
\[a \cdot \xi = \psi(a) \xi \quad \text{and} \quad \xi \cdot b = \xi b \quad \text{and} \quad a \in M \quad , \quad b \in N \quad , \quad \xi \in p\rL^2(N)^\infty \; .\]
Furthermore, if $\psi : M \to p N^\infty p$ and $\eta : M \to q N^\infty q$, then $\bim{M}{\cH(\psi)}{N} \cong \bim{M}{\cH(\eta)}{N}$ if and only if there exists $u \in N^\infty$ satisfying $u u^* = p$, $u^* u = q$ and $\psi(a) = u \eta(a) u^*$ for all $a \in M$.

Note that $M$-$N$-bimodules $\bim{M}{\cH}{N}$ can also be described by means of right actions via $*$-homomorphisms $\psi : N \to p M^\infty p$ as $$\bim{M}{\cH}{N} \cong \bim{M}{\bigl ( (  \ell^2(\N)^* \ot \rL^2(M) )p \bigr )}{\psi(N)} \; \text{.}$$

Let $\cH$ be a right $N$-module and write $\cH_N \cong \bigl (p \eL^2(N)^\infty \bigr)_N$, for a projection $p\in N^\infty$.  Denote $\dim_{\, \text{-}N}(\cH) = (\Tr \ot \tau)(p)$. Observe that the number $\dim_{\, \text{-}N}(\cH)$ depends on the choice of the trace $\tau$, if $N$ is not a factor.

An $M$-$N$-bimodule $\bim{M}{\cH}{N}$ is said to be of \emph{finite Jones index} if $\dim_{M \text{-}}(\cH) < +\infty$ and $\dim_{\, \text{-}N} (\cH) < +\infty$. In particular, the \emph{Jones index} of a subfactor $N \subset M$ is defined as $[M : N]= \dim_{\, \text{-}N} (\eL^2(M))$, see \cite{jones83}. Using the above notations, consider a bimodule of the form $\bim{M}{\cH(\psi)}{N}$ with finite Jones index. Then, one may assume that $\psi$ is a finite index inclusion $\psi : M \droite p N^n p$.

\subsection{Popa's intertwining-by-bimodules technique}

In \cite[Section 2]{popa06-strong-rigidity-1}, Popa introduced a very powerful technique to deduce unitary conjugacy of two von Neumann subalgebras $A$ and $B$ of a tracial von Neumann algebra $M$ from their \emph{embedding} $A \prec_M B$, using \emph{interwining bimodules}. When $A,B \subset M$ are Cartan subalgebras of a II$_1$ factor $M$, Popa proves \cite[Theorem A.1]{popa06-betti-numbers} that $A \prec_M B$ if and only if $A$ and $B$ are actually conjugated by a unitary in $M$. We also recall the notion of \emph{full embedding} $A \prec_M^f B$ of $A$ into $B$ inside $M$.

\begin{definition} \label{def.embed}
Let $M$ be a tracial von Neumann and $A,B \subset M^n$ be possibly non-unital subalgebras. We write
\begin{itemize}
\item
$A \prec_M B$ if $1_A \eL^2(M^n)1_B$ contains a non-zero $A$-$B$-subbimodule $\cK$ that satisfies $\dim_{\, \text{-}B}(\cK) < \infty$.
\item
$A \prec_M^f B$ if $Ap \prec_M B$ for every non-zero projection $p\in 1_A M^n 1_A \cap A'$.
\end{itemize}
\end{definition}

We will use the following characterization of embedding of subalgebras. It can be found in \cite[Theorem 2.1 and Corollary 2.3]{popa06-strong-rigidity-1} (see also Appendix~F in \cite{brownozawa08}).

\begin{theorem}[See \cite{popa06-strong-rigidity-1}]
\label{thm:intertwining}
Let $M$ be a tracial von Neumann algebra and $A, B \subset M^n$ possibly non-unital subalgebras. The following are equivalent.
\begin{itemize}
\item $A \prec_M B$,
\item there exist $m \in \N$, a $*$-homomorphism $\psi : A \ra pB^mp$  and a non-zero partial isometry $v \in 1_A \bigl( \M_{1,m}(\CC) \ot M^n \bigr)p$ satisfying $a v = v \psi(a)$ for all $a \in A$,
\item there is no sequence of unitaries $u_k \in \cU(A)$ such that $\|\rE_B(xu_ky)\|_2 \ra 0$ for all $x,y \in M^n$.
\end{itemize}
\end{theorem}
Note that the entries of $v$ as in in the previous theorem span an $A$-$B$-bimodule $\cK \subset \rL^2(M^n)$ such that $\dim_{\, \text{-}B}( \cK) < \infty$.

We will make use of Theorem \ref{thm.intertw.vaes} due to Vaes, \cite[Theorem 3.11]{vaes08}. We first recall the notion of essentially finite index inclusions of II$_1$ factors (see \cite[Proposition A.2]{vaes08}) and embedding of von Neumann subalgebras inside a bimodule.

Let $N \subset M$ be an inclusion of tracial von Neumann algebras. We say that $N \subset M$ has \emph{essentially finite index} if there exists a sequence of projections $p_n \in N' \cap M$ such that $p_n$ tends to $1$ strongly and $Np_n \subset p_n M p_n$ has finite Jones index for all $n$.

\begin{definition}
Let $M, N$ be tracial von Neumann algebras and $A\subset M$, $B \subset N$ von Neumann subalgebras. Let $\bim{M}{\cH}{N}$  be an $M$-$N$-bimodule. We write
\begin{itemize}
\item
$A \prec_{\cH} B$ if $\cH$ contains a non-zero $A$-$B$-subbimodule $\cK \subset \cH$ with $\dim_{\, \text{-}B}(\cK) < \infty$.
\item
$A \prec_{\cH}^f B$ if every non-zero $A$-$N$-subbimodule $\cK \subset \cH$ satisfies $A \prec_{\cK} B$.
\end{itemize}
\end{definition}

Denote by $\tau$ the trace on $M$. Let $\cH$ be an $M$-$N$-bimodule. Using notations from Section \ref{section.fibimodules}, write $\cH \cong \cH(\psi)$ where $\psi$ is a $*$-homomorphism $\psi: M \droite p N^\infty p$ and $p$ a projection in $N^\infty$. Suppose that $\dim_{\, \text{-}N} (\cH) < +\infty$, i.e $(\Tr \ot \tau)(p) < + \infty$. Then, as remarked in \cite{vaes08}, one has
\begin{itemize}
\item
$A\prec_{\cH} B$ if and only if $\psi(A)\prec_N B$,
\item
$A\prec_{\cH}^f B$ if and only if $\psi(A)\prec_N^f B$.
\end{itemize}

\begin{theorem}[{\cite[Theorem 3.11]{vaes08}}] \label{thm.intertw.vaes}
Let $N, M$ be tracial von Neumann algebras, with trace $\tau$. Let $A \subset M$, $B \subset N$ be von Neumann subalgebras. Assume the following.
\begin{itemize}
\item
Every $A$-$A$-subbimodule $\cK \subset \rL^2(M)$ satisfying $\dim_{\, \text{-}A}(\cK) < +\infty$ is included in $\rL^2(A)$.
\item
Every $B$-$B$-subbimodule $\cK \subset \rL^2(N)$ satisfying $\dim_{\, \text{-}B}(\cK) < +\infty$ is included in $\rL^2(B)$.
\end{itemize}
Suppose that $\bim{M}{\cH}{N}$ is a finite index $M$-$N$-bimodule such that $A \prec_{\cH}^f B$ and  $A \succ_{\cH}^f B$. Then there exists a projection $p \in B^\infty$ satisfying $(\Tr \ot \tau)(p) < +\infty$ and a $*$-homomorphism $\psi : M \to p N^\infty p$ such that
\[\bim{M}{\cH}{N} \cong \bim{M}{\cH(\psi)}{N} \qquad , \qquad \psi(A) \subset  p B^\infty p \]
and this last inclusion has essentially finite index.
\end{theorem}

\subsection{Amalgamated free products of tracial von Neumann algebras} \label{sec.amal}
Throughout this section we consider von Neumann algebras $M_0, M_1$ endowed with faithful normal tracial states $\tau_0,\tau_1$. Let $N$ be a common von Neumann subalgebra of $M_0$ and $M_1$ such that the traces $\tau_0$ and $\tau_1$ coincide on $N$. We denote $M = M_0 *_N M_1$ the amalgamated free product of $M_0$ and $M_1$ over $N$ with respect to the trace preserving conditional expectations (see \cite{popa93} and \cite{voiculescudykemanica92}). Recall that $M$ is endowed with a conditional expectation $E: M \droite N$ and the pair $(M,E)$ is unique up to $E$-preserving isomorphism. The von Neumann algebra
$M_0 *_N M_1$ is equipped with a trace defined by $\tau={\tau_0}\circ E = {\tau_1}\circ E $.

\subsubsection{Rigid subalgebras} Kazhdan's property (T) was generalized to tracial von Neumann algebras by Connes and Jones in \cite{connesjones85} and is defined as follows. A II$_1$ factor $M$ has property (T) if and only if there exists $\epsilon >0$ and a finite subset $F\subset M$ such that every $M$-$M$-bimodule that has a unit vector $\xi$ satisfying $\|x \xi - \xi x\| \leq \epsilon$, for all $x \in F$, actually has a non-zero vector $\xi_0$ satisfying $x \xi_0 = \xi_0 x$, for all $x \in M$.

Note that a group $\Gamma$ in which every non-trivial conjugacy class is infinite (ICC group) has property (T) in the sense of Kazhdan if and only if the II$_1$ factor $\rL(\Gamma)$ has property (T) in the sense of Connes and Jones.

Popa defined a notion of relative property (T) for inclusions of tracial von Neumann algebras, see \cite[Definition 4.2]{popa06-betti-numbers}. Such an inclusion is also called \emph{rigid}. In particular, if $N$ is a II$_1$ factor having property (T), then any inclusion $N \subset M$ in a finite von Neumann algebra $M$ is rigid.

We will make use of the following characterization of relative property (T).
\begin{theorem}[See \cite{popa06-betti-numbers} and \cite{petersonpopa05}]
\label{thm:characterization-property-T}
  An inclusion $N \subset M$ of tracial von Neumann algebras is rigid if and only if every sequence $(\psi_n)$ of trace preserving, completely positive, unital maps $\psi_n:M \ra M$ converging to the identity pointwise in $\twonorm$, converges uniformly in $\twonorm$ on the unital ball $(N)_1$ of $N$.
\end{theorem}

We recall Ioana, Peterson and Popa's Theorem 5.1 from \cite{ioanapetersonpopa08} which controls the position of rigid subalgebras of amalgamated free product von Neumann algebras. We choose to work with matrices over amalgamated free products, which is not a more general situation, since $(M_0 *_N M_1)^n$ can be identified with $M_0^n *_{N^n} M_1^n$.

\begin{theorem}[See {\cite[Theorem 4.3]{ioanapetersonpopa08}}] \label{IPP1}
Let $M = M_0 *_N M_1$. Let $p \in M^n$ be a projection and $Q \subset p M^n p$ a rigid inclusion. Then there exists $i \in \{0,1\}$ such that $Q \prec_M M_i$.
\end{theorem}

\subsubsection{Control of quasi-normalizers} Let $M$ be a tracial von Neumann algebra and $N \subset M$ a von Neumann subalgebra. The \emph{quasi-normalizer} of $N$ inside $M$, denoted $\QN_M(N)$, is defined as the set of elements $a \in M$ for which there exist  $a_1,\dots, a_n, b_1,\dots,b_m \in M$ such that
\[ Na \subset \sum_{i=1}^n a_i N , \qquad , \qquad aN \subset \sum_{i=1}^m N b_i \; \text{.} \]
The inclusion $N \subset M$ is called \emph{quasi-regular} if $\QN_M(N)'' = M$. One also defines the \emph{group of normalizing unitaries} $\norm_N(M)$ of $N\subset M$ as the set of unitaries $u \in M$ satisfying $u N u^* = N$. The \emph{normalizer} of $N$ in $M$ is $\norm_M(N)''$. Note that $ N' \cap M  \subset \norm_M(N)'' \subset \QN_M(N)''$.

%\lspan \bigl (\norm_N(M) \bigl ) \subset \QN_N(M)$.

\begin{theorem}[See {\cite[Theorem 1.1]{ioanapetersonpopa08}}] \label{IPP2}
Let $M = M_0 *_N M_1$. Let $p \in M_0^n$ be a projection and $Q \subset p M_0^n p$ a von Neumann subalgebra satisfying $Q \not\prec_{M_0} N$. Whenever $\cK \subset p(\CC^n \ot \rL^2(M))$ is a $Q$-$M_0$-subbimodule with $\dim_{\, \text{-}M_0}(\cK) < +\infty$, we have $\cK \subset p(\CC^n \ot \rL^2(M_0))$. In particular, the quasi-normalizer of $Q$ inside $p M^n p$ is contained in $p M_0^n p$.
\end{theorem}

\subsection{Tensor \Cstar-categories, fusion algebras and bimodule categories of II$_1$ factors}
\label{sec.fusion}

We briefly recall some definitions for tensor \Cstar-categories and refer to \cite{longoroberts97, schaflitzel97} for more information and precise statements. A tensor \Cstar-category is a \Cstar-category with a \emph{monoidal structure}, such that all structure maps are unitary. A tensor \Cstar-category is called \emph{regular} if it has subobjects and direct sums and the unit object is strongly irreducible. A regular tensor \Cstar-category is called \emph{compact} if every object has a conjugate. A compact tensor \Cstar-category is \emph{finite} if it has only finitely many isomorphism classes of simple objects.

{\bf Convention.} Throughout this article we assume without loss of generality that all tensor categories involved are strict.

\subsubsection{Fusion algebras}
A fusion algebra $\cA$ is a free $\N$-module $\N[\cG]$ equipped with
\begin{itemize}
\item an associative and distributive product operation, and a multiplicative unit element $e \in \cG$,
\item an additive, anti-multiplicative, involutive map $x \mapsto \overline{x}$, called \emph{conjugation},
\end{itemize}
satisfying Frobenius reciprocity as follows. For $x,y,z \in \cG$, define $m(x,y;z) \in \N$ such that $$x y = \sum_z m(x,y;z) z\;.$$ Then,
one has $m(x,y;z) = m(\ox,z;y) = m(z,\overline{y}; x)$ for all $x,y,z \in \cG$.

The base $\cG$ of the fusion algebra $\cA$, also called the \emph{irreducible elements} of $\cA$, consists of the non-zero elements of $\cA$ that cannot be expressed as the sum of two non-zero elements.

We have the following examples of fusion algebras.
\begin{itemize}
  \item Given a countable group $\Gamma$, one gets the associated fusion algebra $\cA = \N[\Gamma]$.
  \item Let $G$ be a locally compact group and define the fusion algebra $\cA$ of $\Repfin(G)$ as the set of equivalence classes of finite dimensional unitary representations of $G$. The direct sum and tensor product of representations in $\Repfin(G)$ yield a fusion algebra structure on $\cA$.
  \item More generally, the isomorphism classes of objects in a compact tensor \Cstar-category form a fusion algebra. Note that there exist non-equivalent tensor \Cstar-categories having isomorphic fusion algebras.
\end{itemize}

In this article we are mainly interested in tensor \Cstar-categories and fusion algebras coming from bimodules over II$_1$ factors. We recall some definitions and refer to \cite{bisch97} for background material and results on bimodules and fusion algebras, in particular in relation with subfactors.

\subsubsection{The bimodule category of a II$_1$ factor}

Let  $M$, $N$, $P$ be II$_1$ factors. We denote by $\cH \ot_N \cK$ the Connes tensor product of the $M$-$N$-bimodule $\cH$ and the $N$-$P$-bimodule $\cK$ and refer to \cite[V.Appendix B]{connes94} for details. Note that $\cH(\rho) \ot_N \cH(\psi) \cong \cH((\id \ot \psi)\rho)$.

We recall now the following useful lemma from \cite{falguieresvaes10} concerning Connes tensor product versus product in a given module. The inclusion of II$_1$ factors $N \subset M$ considered in \cite{falguieresvaes10} is assumed to be irreducible ($N' \cap M = \C1$). Instead, we assume that $N \subset M$ is quasi-regular. We give a proof for the convenience of the reader.

\begin{lemma}[{\cite[Lemma 2.2]{falguieresvaes10}}]\label{lemm.connes_tensor}
Let $\wt{N} \subset N \subset M$ be an inclusion of II$_1$ factors and let $P$ be a II$_1$ factor. Assume that $N \subset M$ is quasi-regular and $\wt{N} \subset N$ has finite index. Let $\bim{M}{\cH}{P}$ be an $M$-$P$-bimodule. Suppose that $\cL \subset \cH$ is a closed $\wt{N}$-$P$-subbimodule. Suppose that $\cK \subset \rL^2(M)$ is an $N$-$\wt{N}$-subbimodule of finite index.
Denote by $\cK \cdot \cL$ the closure of $(\cK \cap M) \cL$ inside $\cH$. Then
\begin{itemize}
\item
$\cK \cdot \cL$ is an $N$-$P$-bimodule isomorphic to a subbimodule of $\cK \ot_{\wt{N}} \cL$.
\item
If $\cK \cdot \cL$ is non-zero and $\cK \ot_{\wt{N}} \cL$ is irreducible then, $\cK \cdot \cL$ and $\cK \ot_{\wt{N}} \cL$ are isomorphic $N$-$P$-bimodules.
\end{itemize}

Whenever $\bim{P}{\cH}{M}$ is a $P$-$M$-bimodule with closed $P$-$\wt{N}$-subbimodule $\cL$ and $\cK \subset \Ltwo(M)$ an $\wt{N}$-$N$-subbimodule, we define $\cL \cdot \cK$ as the closure of $\cL (\cK \cap M)$ inside $\cH$ and, by symmetry, we find that $\cL \cdot \cK$ is isomorphic with a $P$-$N$-subbimodule of $\cL \ot_N \cK$.
\end{lemma}
\begin{proof}
  Let $\cH, \cK$ and $\cL$ be as in the statement of the lemma.  Note that $\cK \cap M$ is dense in $\cK$, since $N \subset M$ is quasi-regular and $\wt{N} \subset N$ has finite index. Moreover, all vectors in $\cK \cap M$ are $\wt{N}$-bounded. So, there exists a finite index inclusion $\psi:N \ra p\wt{N}^n p$ and an $N$-$\wt{N}$-bimodular isomorphism $T:\cH(\psi) = p(\CC^n \otimes \Ltwo(\wt{N})) \ra \cK$ such that $T(p(e_i \otimes 1)) \in \cK \cap M$ for all $i$. We have $\cK \otimes_{\wt{N}} \cL \cong p(\CC^n \otimes \cL)$, hence we can define an $N$-$P$-bimodular map $S:p(\CC^n \otimes \cL) \ra \cK \cdot \cL$ by $S(p(e_i \otimes \xi)) = T(p(e_i \otimes 1)) \cdot \xi$. The range of $T$ is dense in $\cK \cdot \cL$. After taking the polar decomposition of $T$ we get a coisometry $\cK \otimes_N \cL \ra \cK \cdot \cL$.
\end{proof}

The \emph{contragredient} of an $M$-$N$-bimodule $\bim{M}{\cH}{N}$ is the $N$-$M$-bimodule defined on the conjugate Hilbert space $\overline{\cH}$ with bimodule actions given by $a \cdot \overline{\xi} = \overline{(\xi a^*)}$ and $\overline{\xi} \cdot b = \overline{(b^* \xi)}$.

The Connes tensor product and contragredience induce a compact tensor \Cstar-category structure on the category of finite index $M$-$M$ bimodules, where morphisms are given by bimodular maps.

\begin{definition}
Let $M$ be a II$_1$ factor. We define $\Bimod(M)$ to be the tensor \Cstar-category of finite index $M$-$M$-bimodules and  $\FAlg(M)$ the associated fusion algebra.
\end{definition}

We recall the notion of pairs of conjugates in strict tensor \Cstar-categories.

\begin{definition}[See \cite{longoroberts97}]
Let $x$ be an object in a strict tensor \Cstar-category $\cC$. A conjugate for $x$ is an object $\ol{x}$ in $\cC$ and morphisms $R: 1_\cC \ra \ol{x} \otimes x$, $\ol{R}: 1_\cC \ra x \otimes \ol{x}$ such that
\[(\ol{R}^* \otimes \id_x) \circ (\id_x \otimes R) = \id_x
   \quad \text{and} \quad
   (R^* \otimes \id_{\ol{x}}) \circ (\id_{\ol{x}} \otimes \ol{R}) = \id_{\ol{x}} \eqstop\]
\end{definition}

In the following theorem, pairs of conjugates are used to characterize finite index bimodules among all bimodules over a II$_1$ factor (see \cite{longoroberts97} and also \cite[Theorem 5.32]{falguieres09}).

\begin{theorem} \label{thm.caracterisatio-bimod}
Let $M$ be a II$_1$ factor and let $\bim{M}{\cH}{M}$ an $M$-$M$-bimodule. Then $\bim{M}{\cH}{M}$ has finite index if and only if $\bim{M}{\cH}{M}$ has a conjugate in the tensor \Cstar-category of all $M$-$M$-bimodules.
\end{theorem}

\subsubsection{Tensor \Cstar-categories arising from subfactors} \label{sec.Bimod_subfactor}
Let $M$ be a II$_1$ factor and $N \subset M$ a subfactor. Write $e_N$ for the projection $\rL^2(M) \to \rL^2(N)$. The von Neumann algebra $\psg M, e_N \psd \subset \B( \rL^2(M))$ generated by $M$ and $e_N$, called the \emph{Jones basic construction}, was introduced in \cite{jones83} and is denoted $M_1$. Note that $\rL^2(M_1)$ is an $M$-$M$-bimodule and it is of finite Jones index whenever $[M:N] < + \infty$. We will frequently use the fact that $\dimvec (N' \cap M) < + \infty$ if $[M:N] < + \infty$.

\begin{definition}
Let $N \subset M$ be an inclusion of type II$_1$ factors. We define $\Bimod(N \subset M)$ to be the tensor \Cstar-subcategory of $\Bimod(N)$ generated by all finite index $N$-$N$-bimodules that appear in $\eL^2(M)$. We denote by $\FAlg(N \subset M)$ the associated fusion subalgebra of $\FAlg(N \subset M)$.
\end{definition}

We give the following definition of depth 2, as in \cite{enockvallin00}.

\begin{definition}
   Let $N \subset Q$ be an inclusion of II$_1$ factors.  Let $N \subset Q \subset Q_1 \subset Q_2 \subset \dotsm$ be the Jones tower.  Then $N \subset Q$ has depth 2 if $N' \cap Q \subset N' \cap Q_1 \subset N' \cap Q_2$ is a basic construction.
\end{definition}

Identify $N' \cap Q_2$ with the space of bounded $N$-$Q$-bimodular maps $\rB_{N \text{-} Q}(\Ltwo(Q_1, \Tr))$. Denote by $\Hom_{N\text{-}Q} (\Ltwo(Q), \Ltwo(Q_1))$ the Hilbert space completion of $N$-$Q$-bimodular maps from $\Ltwo(Q, \tau)$ to $\Ltwo(Q_1,\Tr)$ with respect to the scalar product $\langle T , S \rangle = \tau(S^*T)$. We recall the following special case of \cite[Theorem 3.10]{enockvallin00}.

\begin{theorem}[{See \cite[Theorem 3.10]{enockvallin00}}]
\label{sec:depth-2:thm:characterization}
The inclusion $N \subset Q$ of II$_1$ factors has depth 2 if and only if the natural action of $N' \cap Q_2$ on $\Hom_{N \text{-}Q} (\Ltwo(Q), \Ltwo(Q_1))$ is faithful.
\end{theorem}

As a consequence, we obtain the following characterization of depth 2 inclusions that we use in this article.

\begin{corollary} \label{corol.depth2}
Let $N,Q$ be II$_1$ factors. Then, the inclusion $N \subset Q$ has depth 2 if and only if $\bim{N}{\Ltwo(Q_1)}{Q}$ is isomorphic to an $N$-$Q$-subbimodule of $\bim{N}{\Ltwo(Q)^{\infty}}{Q}$.
\end{corollary}
\begin{proof}
Let $N \subset Q$ be a depth 2 inclusion of II$_1$ factors. Let $p \in N' \cap Q_2$ be the projection onto the orthogonal complement of the maximal $N$-$Q$-subbimodule of $\Ltwo(Q_1)$ which is contained in $\bim{N}{\Ltwo(Q)^{ \infty}}{Q}$.  Then, $p$ acts trivially on $\Hom_{N \text{-}Q} (\Ltwo(Q), \Ltwo(Q_1))$. Therefore, $p = 0$ by Theorem \ref{sec:depth-2:thm:characterization}.

Assume that $\bim{N}{\Ltwo(Q_1)}{Q}$ is isomorphic to a subbimodule of $\bim{N}{\Ltwo(Q)^{ \infty}}{Q}$.  Let $p \in N' \cap Q_2$ be a non-zero projection. Then $p\Ltwo(Q_1)$ is a non-zero $N$-$Q$-bimodule, so there is a non-trivial $N$-$Q$-bimodular map $T:\Ltwo(Q) \ra p\Ltwo(Q_1)$. We have $p \cdot T = T \neq 0$, so $p$ acts non-trivially on $\Hom_{N \text{-}Q}(\Ltwo(Q), \Ltwo(Q_1))$. We have proven that $N' \cap Q_2$ acts faithfully on $\Hom_{N \text{-}Q}(\Ltwo(Q), \Ltwo(Q_1))$. We conclude using again Theorem \ref{sec:depth-2:thm:characterization}.
\end{proof}

\subsubsection{The fusion algebra of almost-normalizing bimodules} Let $N \subset M$ be a regular inclusion, i.e. $\norm_M(N)'' = M$. For any element $u \in \norm_M(N)$ the $N$-$N$-bimodule $u\Ltwo(N)$ has finite index and lies in $\Ltwo(M)$. Such bimodules are generalized by the notion of  bimodules almost-normalizing the inclusion $N \subset M$, which was introduced by Vaes in \cite{vaes09}. This notion was adapted to more general irreducible, quasi-regular inclusions of II$_1$ factors $N \subset M$ in \cite{falguieresvaes10}. We recall the definition.

\begin{definition}
Let $N \subset M$ be an irreducible and quasi-regular inclusion of type II$_1$ factors. A finite index $N$-$N$-bimodule is said to almost-normalize the inclusion $N \subset M$, inside $\FAlg(N)$, if it arises as a finite index $N$-$N$-subbimodule of a finite index $M$-$M$-bimodule. We denote by $\ANFAlg(N \subset M)$ the fusion algebra generated by $N$-$N$-bimodules almost-normalizing the inclusion $N \subset M$.
\end{definition}

Let $N$ be a II$_1$ factor and $\Gamma$ a countable group acting outerly on $N$.  Write $M = N \rtimes \Gamma$ and assume that the inclusion $N \subset N \rtimes \Gamma$ is rigid. It is proven in \cite[Lemma 4.1]{vaes09} that the fusion algebra $\ANFAlg(N \subset N \rtimes \Gamma)$ is a countable fusion subalgebra of $\FAlg(N)$. The next lemma is a straightforward adaptation of \cite[Lemma 4.1]{vaes09}.

\begin{lemma}\label{lem.countablefusionalgebra}
Let $N \subset M$ be a rigid, irreducible and quasi-regular inclusion of type II$_1$ factors. Then, the fusion algebra $\ANFAlg(N\subset M)$ is a countable fusion subalgebra of $\FAlg(N)$.
\end{lemma}

\subsubsection{Freeness of fusion algebras}
The notion of freeness of fusion algebras was introduced in \cite[Section 1.2]{bischjones97}, in the study of free composition of subfactors. We recall the definition.

\begin{definition}[{\cite[Section 1.2]{bischjones97}}] \label{def.freeness}
Let $\cA$ be a fusion algebra and $\cA_0, \cA_1 \subset \cA$ fusion subalgebras. We say that $\cA_0$ and $\cA_1$ are \emph{free inside $\cA$} if every alternating product of irreducibles in $\cA_i \setminus \{e\}$, remains irreducible and different from $e$.
\end{definition}

Let $M$ be a II$_1$ factor and $\bim{M}{\cK}{M}$ a finite index $M$-$M$-bimodule. Whenever $\al \in \Aut(M)$, we define the conjugation of $\cK$ by $\al$ as the bimodule $\cK^\al = \cH(\al^{-1}) \ot_M \cK \ot_M \cH(\al)$. Denote by $R$ the hyperfinite II$_1$ factor. Vaes proved in \cite[Theorem 5.1]{vaes09} that countable fusion subalgebras of $\FAlg(R)$ can be made free by conjugating one of them with an automorphism of $R$ (see Theorem \ref{theo.free} below). Note that the same result has first been proven for countable subgroups of $\Out(R)$ in \cite{ioanapetersonpopa08}. In both cases, the key ingredients come from \cite{popa95}.

\begin{theorem}[{\cite[Theorem 5.1]{vaes09}}]\label{theo.free}
Let $R$ be the hyperfinite II$_1$ factor and $\cA_0,\cA_1$ two countable fusion
subalgebras of $\FAlg(R)$. Then,
$$\{\alpha \in \Aut(R) \mid  \cA_0^{\alpha}\ \text{and}\ \cA_1\ \text{are free}\}$$
is a dense $G_{\delta}$-subset of $\Aut(R)$.
\end{theorem} 

%%% Local Variables: 
%%% mode: latex
%%% TeX-master: "construction-bimodule-categories-hyperfinite-case"
%%% End: 

%% file: main-part4.tex
\section{Proof of Theorem \ref{thm:construct-bimodule-categories}}
\label{section.thmC}

We recall the following construction, from \cite{falguieresvaes10}. Consider the group $\Gamma = \Q^3 \oplus \Q^3 \rtimes \SL(3,\Q)$, defined by the action $A \cdot (x,y) = (Ax , (A^t)^{-1} y)$ of $\SL(3,\Q)$ on $\Q^3 \oplus \Q^3$. Take $\al \in \RR - \QQ$, define $\Omega_\alpha \in \rZ^2(\Q^3\oplus \Q^3,S^1)$ such that
  \begin{align*}
      \Omega_\alpha\bigl( (x,y),(x',y') \bigr) & =
             \exp\bigl(2\pi i\al (\langle x,y'\rangle - \langle y, x'\rangle)\bigr)
                    \qquad\text{for all}\;\; (x,y), (x',y') \in \Q^3 \oplus \Q^3 \eqcomma
  \end{align*}
and extend $\Omega_\alpha$ to an $S^1$-valued $2$-cocycle on $\Gamma$ by $\SL(3,\Q)$-invariance. Write $\Lambda = \ZZ^3 \oplus \ZZ^3$. Then, by \cite[Lemma 3.3]{falguieresvaes10} and \cite[Example 3.4]{falguieresvaes10}, the inclusions $N \subset N_0 \subset P$ given by
\[N = \rL_{\Omega_\alpha}(\Lambda), \quad N_0
    = \rL_{\Omega_\alpha}(\Z^3 \oplus \Z^3 \rtimes \SL(3,\Z)), \quad P
    = \rL_{\Omega_\alpha}(\Gamma)\]
satisfy the following properties.
\begin{itemize}
\item[$(\cP_1)$]
$N \subset P$ is irreducible and quasi-regular,
\item[$(\cP_2)$] $N_0 \subset P$ is quasi-regular,
\item[$(\cP_3)$] $N_0$ has property $(T)$.
\end{itemize}
Note that $(\cP_1)$ follows from the fact that the inclusion $\Lambda \subset \Gamma$ is almost-normal, meaning the commensurator $\Comm_\Gamma(\Lambda)$ defined as
\[\Comm_\Gamma(\Lambda) :=\{ g \in \Gamma \mid g \Lambda g^{-1} \cap \Lambda\ \textrm{has finite index in}\ g\Lambda g^{-1} \textrm{and in}\ \Lambda \}\]
 is the whole of $\Gamma$. We know that the group $\SL(3,\Q)$ does not have any non-trivial finite dimensional unitary representations (see \cite{vonneumannwigner40}). The smallest normal subgroup of $\Gamma$ containing $\SL(3,\Q)$ is $\Gamma$ itself. This gives the following property.
\begin{itemize}
\item[$(\cP_4)$] The group $\Gamma$ has no non-trivial finite dimensional unitary representations.
\end{itemize}
We will also need the following additional property, proven in \cite[Example 3.4]{falguieresvaes10}.
\begin{itemize}
\item[$(\cP_5)$] The inclusion $\rL_{\Omega_\alpha}(\Lambda_0) \subset \rL_{\Omega_\alpha}(\Gamma)$ is irreducible, for every finite index subgroup $\Lambda_0 < \Lambda$.
\end{itemize}

\begin{theorem} \label{thm.main}
Let $Q$ be a II$_1$ factor such that $N \subset Q$. Let $B = N' \cap Q$ and assume that
\begin{itemize}
\item
$N \subset Q$ is a quasi-regular and depth $2$ inclusion,
\item
$B$ is hyperfinite,
\item
there is no non-trivial $*$-homomorphism from $N_0$ to any amplification of $Q$,
\item
the fusion algebras $\ANFAlg(N \subset P)$ and $\FAlg(N \subset Q)$ defined in Section \ref{sec.fusion} are free inside $\FAlg(N)$.
\end{itemize}
Then, for $M = \bigl( P \vNtensor B \bigr )*_{N \vnt B} Q$, we have that $\Bimod(M) \simeq \Bimod(Q \subset Q_1)$, as tensor \Cstar-categories, where $N \subset Q \subset Q_1$ is the basic construction.
\end{theorem}

{\bf Outline of the proof of Theorem \ref{thm:construct-bimodule-categories}.} We first prove Theorem \ref{thm.main} in two steps. In Section \ref{sec:functor}, we construct a fully faithful tensor \Cstar-functor $F : \bimod(Q \subset Q_1) \to \bimod(M)$. In Section \ref{sec:F-surjective}, we prove that $F$ is essentially surjective, which completes the proof of Theorem \ref{thm.main}. In Section \ref{proof.construct-bimodule-categories}, we give a proof of Theorem \ref{thm:construct-bimodule-categories}, relying on Theorem \ref{thm.main}.

In the rest of Section \ref{section.thmC} we will always use the notations of Theorem \ref{thm.main}.

\subsection{A fully faithful functor $F : \bimod(Q \subset Q_1) \to \bimod(M)$}
\label{sec:functor}

Denote by $\cC$ the tensor \Cstar-category whose objects are finite index inclusions $\psi: Q \ra pQ^\infty p$ with $p \in B^\infty$, $(\Tr \otimes \tau)(p) < \infty$ and $\psi(x) = xp$ for all $x \in N$. The tensor product on $\cC$ is given by $\psi_1 \otimes_{\cC} \psi_2 = (\id \otimes \psi_2) \circ \psi_1$. Morphisms of $\cC$ are given by
\[\mathrm{Hom}_{\cC}(\psi_1, \psi_2) = \{T \in q B^\infty p \amid \forall x \in Q: \,T\psi_1(x) =\psi_2(x) T \} \eqstop\]

\begin{proposition}
\label{prop:representants-for-bimodules}
  The natural inclusion $I: \psi \mapsto \cH(\psi)$ of $\cC$ into $\bimod(Q)$ defines an equivalence of tensor \Cstar-categories $\cC \simeq \bimod(Q \subset Q_1)$.
\end{proposition}
\begin{proof}
It is easy to check that $I$ is a faithful tensor \Cstar-functor. We prove that $I$ is full and that its essential range is $\bimod(Q \subset Q_1)$.

We first prove that $I$ is full. Let $T: p \Ltwo(Q)^{\infty} \ra q \Ltwo(Q)^{\infty}$ be a $Q$-$Q$-bimodular map between $\cH(\psi_1)$ and $\cH(\psi_2)$. Then $T \in pQ^\infty q$, since $T$ is right $Q$-modular. We have $Txp = xqT$ for all $x \in N$, so it follows that $T \in pB^\infty q$. This proves that $I$ is full.

Let us prove that the image of $I$ is contained in $\bimod(Q \subset Q_1)$. Take a finite index inclusion $\psi: Q \ra pQ^\infty p$ with $p \in B^\infty$, $(\Tr \otimes \tau)(p) < \infty$ and $\psi(x) = xp$ for all $x \in N$ and let $\cH = \cH(\psi)$. We claim that $\cH$ is a $Q$-$Q$-subbimodule of $\eL^2(Q_1)^\infty$. Extend $\psi$ to a map $\rL^2(Q) \to \rL^2(p Q^\infty p)$ and note that its entries, considered as operators on $\rL^2(Q)$, lie in $Q_1$. Any non-zero column of $\psi$ defines a partial isometry $v \in p(\Cmatnm{\infty}{1} \vnt Q_1)$ satisfying $vx = \psi(x)v$, for all $x \in Q$. Note that $vv^* \in \psi(Q)' \cap p Q_1^\infty p$. If $p\neq vv^*$, then we may apply the previous procedure to the non-zero $Q$-$Q$-bimodule $(p - vv^*) \cdot \cH \cong \cH(\psi(\cdot)(p - vv^*))$. Take a maximal family of non-zero partial isometries $v_i$ inside $p(\Cmatnm{\infty}{1} \vnt Q_1)$ satisfying $\psi(x)v_i = v_i x$ for all $x \in Q$ and such that $v_i v_i^*$ are pairwise distinct orthogonal projections. Consider the projection $r = p- \sum v_i v_i^*$. If $r \neq 0$ then we can apply the previous procedure to the non-zero bimodule $r \cdot \cH$. As above, we get a non-zero partial isometry $w \in r (\Cmatnm{\infty}{1} \vnt Q_1)$ such that $\psi(x)w=wx$, for all $x \in Q$. Then, $ww^*$ is orthogonal to all of the $v_i v_i^*$, which contradicts maximality of the family. So, $p = \sum v_i v_i^*$.  Putting all these partial isometries in a row, we get an element $u \in p (Q_1)^\infty $ such that $u x = \psi(x) u$, for all $x \in Q$ and $uu^*= \sum v_i v_i^* = p$. This proves our claim.

We now prove that every bimodule $\cH$ of $\bimod(Q \subset Q_1)$ is contained in the essential range of $I$. Assume that $\cH$ arises as a $Q$-$Q$-subbimodule of $\Ltwo(Q_1)^{\otimes_Q k}$, for some $k \in \N$. We prove that $\cH$ is a subbimodule of $\Ltwo(Q_1)^\infty$. By Corollary \ref{corol.depth2}, we have that $\cH$ is isomorphic, as $N$-$Q$-bimodule, to a subbimodule of $\eL^2(Q)^\infty$. Writing $\cH \cong \cH(\psi)$, for some finite index inclusion $\psi: Q \ra q Q^n q$, we find a non-zero $N$-central vector $v \in q(\Cmatnm{n}{1} \otimes \Ltwo(Q))$. Taking polar decomposition, we may assume that $v \in q(\Cmatnm{n}{1} \ot Q)$ is a partial isometry satisfying $\psi(x)v = vx$, for all $x \in N$. As a consequence, we have $v^*v \in B$. As in the previous paragraph, take a maximal family of non-zero partial isometries $v_i$ inside $q(\Cmatnm{n}{1} \ot Q)$ satisfying $\psi(x)v_i = v_i x$ for all $x \in N$ and $q = \sum v_i v_i^*$. Putting all partial isometries $v_i$ in one row, we get an element $u \in q(\Cmatnm{n}{\infty} \vnt Q)$ such that $\psi(x)u = u x$ for all $x \in N$ and $uu^* = \sum v_i v_i^* = q$. Define $p = u^* u$ and note that $p \in B^\infty$. Conjugating with $u^*$ from the beginning yields a map $\psi : Q \to p Q^\infty p$ such that $\psi(x) = px$, for all $x \in N$ and still satisfying $\cH \cong \cH(\psi)$.
\end{proof}

Take a finite index inclusion $\psi: Q \ra pQ^{\infty}p$ in $\cC$. Then, we have $p \in B^\infty$. Denote by $\iota : P \vnt B \ra p(P \vnt B)^\infty p$ the inclusion map given by $x \mapsto xp$ on $P$ and by the restriction $\psi|_B$ on $B$. Since $\psi$ preserves $N$, it also preserves $B = N' \cap Q$ and we obtain a map $\iota * \psi :  M \ra pM^\infty p$. If $T \in \mathrm{Hom}_{\cC}(\psi_1,\psi_2)$, then $T \in q B^\infty p$. So, $T$ defines an $M$-$M$-modular map from $\cH(\iota * \psi_1)$ to $\cH(\iota * \psi_2)$. We conclude that the map
\[F_0: \cC \to \bimod(M): \psi \mapsto \cH(\iota * \psi)\]
is a functor.

\begin{proposition}
$F_0$ is a fully faithful tensor \Cstar-functor.
\end{proposition}
\begin{proof}
It is clear that $F_0$ is faithful. We first prove that $F_0$ is full. Take $T \in \Hom_{M\text{-}M} (\cH(\iota * \psi_1), \cH(\iota * \psi_2))$.  Then $T: p \Ltwo(M)^\infty \ra q \Ltwo(M)^\infty$ is right $M$-modular, hence $T \in p M^\infty q$.  Since $Txp = xqT$ for all $x \in P$, we have $T \in p B^\infty q$.  So $T$ is in the image of $F_0$. The functor $*$ on both $\bimod(Q \subset Q_1)$ and $\bimod(M)$ is given by $T \mapsto T^*$, so $F$ is a \Cstar-functor. Since $\cH(\psi_1) \otimes_M \cH(\psi_2) \cong \cH((\id \otimes \psi_2)\circ \psi_1)$, it follows immediately that $F_0$ is a tensor $\Cstar$-functor.
\end{proof}

Now let $G: \bimod(Q \subset Q_1) \ra \cC$ be an inverse functor for the inclusion $I: \cC \ra \bimod(Q \subset Q_1)$. We define the fully faithful tensor \Cstar-functor $F = F_0 \circ G: \bimod(Q \subset Q_1) \ra \bimod(M)$.

\subsection{Proof of Theorem \ref{thm.main}: essential surjectivity of $F$}
\label{sec:F-surjective}

We give a series of preliminary lemmas before proving that the functor $F$ constructed in the previous section is essentially surjective.

\begin{lemma}\label{lem:fi-N-bimodule}
Let $\bim{M}{\cH}{M}$ be a finite index $M$-$M$-bimodule and $\bim{P \vNtensor B}{\cK}{P \vNtensor B} \subset \bim{P \vNtensor B}{\cH}{P \vNtensor B}$ be a finite index $P \vNtensor B$-$P \vNtensor B$-subbimodule. Then $\cK$ contains a non-zero $N$-$N$-subbimodule $\cL$ such that $\dim_{\, \text{-} N}(\cL) < + \infty$.
\end{lemma}

\begin{proof}
Let $\psi : M \to p M^n p$ and $\vfi : P \vNtensor B \to q (P \vNtensor B)^k q$ be finite index inclusions such that $\bim{M}{\cH}{M} \cong \bim{M}{\cH(\psi)}{M}$ and $\bim{P \vNtensor B}{\cK}{P \vnt B} \cong \bim{P \vnt B}{\cH(\vfi)}{P \vnt B}$. Take a non-zero partial isometry $v_0 \in p(\rM_{n,k}(\C) \ot M)q$ such that $\psi(x)v_0 = v_0 \vfi(x)$, for all $x \in P \vNtensor B$.  We have $v_0^*v_0 \in \vfi(P \otimes B)' \cap qM^kq$, so the support projection $\supp \rE_{P \vnt B}(v_0^*v_0)$ lies in $\vfi(P \vnt B)' \cap q(P \vnt B)^k q$.  Moreover $v_0 (\supp \rE_{P \vnt B}(v_0^*v_0)) = v_0$. So we can assume that $q = \supp \rE_{P \vnt B}(v_0^*v_0)$.

We claim that $\vfi(N_0) \prec_{P \vnt B} P$. Recall that $B$ is hyperfinite, by assumption. Let $\bigcup_n A_n$ be the dense union of an increasing sequence of finite dimensional von Neumann subalgebras $A_n$ of $B$. Since $P \otimes 1 \subset P \otimes A_n$ is a finite index inclusion for every $n$, it suffices to show that $\vfi(N_0) \prec_{P \vnt B} P \otimes A_n$ for some $n$. Denote by $\rE_n$ the trace-preserving conditional expectation of $B$ onto $A_n$. Then the sequence of unital completely positive maps $\id \otimes \rE_n$ on $(P \vnt B)^k$, still denoted by $\rE_n$, converges pointwise in $\twonorm$ to the identity. Since $N_0$ has property (T) (see $(\cP_3)$), Theorem \ref{thm:characterization-property-T} shows that $(\rE_n)$ converges uniformly in $\twonorm$ on $(\vfi(N_0))_1$. Take $n \in \NN$ such that $\|\rE_n(x) - x\|_2 < 1/2$ for all $x \in \vfi(N_0)$. Assume that $\vfi(N_0) \not \prec_{P \vnt B} P \otimes A_n$. By Theorem \ref{thm:intertwining}, there is a sequence of unitaries $u_k \in \cU(\vfi(N_0))$ such that for all $x,y \in q(P \vnt B)^k q$, we have $\|\rE_n(xu_ky)\|_2 \ra 0$ for $k \ra \infty$. In particular,
\[1 = \|u_k\|_2 < 1/2 + \|\rE_n(u_k)\|_2 \ra 1/2 \eqcomma\]
 which is a contradiction. We have proved our claim.

This yields a $*$-homomorphism $\pi: N_0 \ra r P^l r$ and a non-zero partial isometry $v_1 \in q(\rM_{k,l}(\C) \ot P\vnt B)r$ such that $\vfi(x)v_1 = v_1 \pi(x)$, for all $x \in N_0$. Similarly to the first paragraph, we can assume that $r = \supp E_P(v_1^*v_1)$. Note that $\rE_{P \vnt B}(v_0^*v_0) = q$. So $v = v_0 v_1 \in p(\rM_{n,l} \otimes M)r$ is a non-zero partial isometry. Moreover, we have $v\pi(x) = \psi(x)v$, for all $x \in N_0$.

We claim that $\pi(N) \prec_{P} N$. We first prove that it suffices to show that $\pi(N) \prec_{P \vnt B} N \vnt B$. Indeed, if this the case, we get a $*$-homomorphism $\theta : N \to t (N \vnt B)^j t$ and a non-zero partial isometry $u \in r(\M_{l,j}(\C) \ot P \vnt B)t$ such that $\pi(x)u = u \theta(x)$, for all $x \in N$. Denote by $u_i$ the $i$-th column of $u$. Then the closed linear span of $\{u_i N \vnt B \, | \, i = 1, \dotsc ,j\}$ defines a non-zero $\pi(N)$-$N \vnt B$-subbimodule of $r(\CC^l \otimes \Ltwo(P \vnt B))$ with finite right dimension. Using the $N$-$N$-modular projection onto $r(\CC^l \otimes \Ltwo(P))$ and the action of $B$, we find a non-zero $\pi(N)$-$N$-bimodule inside $r(\CC^l \otimes \Ltwo(P))$ which is finitely generated as a right $N$-module. 

Assume now that $\pi(N) \not \prec_{P \vnt B} N \vnt B$.  Then, Theorem \ref{IPP2} implies that the quasi-normalizer of $\pi(N)$ in $r M^l r$ sits inside $r (P \vnt B)^l r$. As a consequence, $v^*v \in \pi(N)' \cap r M^l r \subset r (P \vnt B)^l r$. Since the inclusion $N \subset M$ is quasi-regular, we have that
\begin{equation} \label{eq.psi}
v^* \psi(M) v \subset r (P \vnt B)^l r\; \text{.}
\end{equation}
Denote by $A$ the von Neumann algebra generated by $\psi(M)$ and $vv^*$. Then $\psi(M) \subset A \subset p M^n p$ and $A \subset p M^n p$ has finite index. Using \eqref{eq.psi}, we get that $v^*Av \subset v^*v  (P \vnt B)^l v^* v \subset v^* M^n v$, from which we deduce that $P \vnt B \subset M$ has finite index. We get a contradiction. Indeed, $M$ being an amalgamated free product, we can find in $\Ltwo(M)$ infinitely many pairwise orthogonal $P \vnt B$-$P \vnt B$-bimodules by means of alternating powers of $\Ltwo(P \vnt B) \ominus \Ltwo(N \vnt B)$ and $\Ltwo(Q) \ominus \Ltwo(N \vnt B)$.

The previous claim yields a $*$-homomorphism $\rho: N \rightarrow s N^m s$ and a non-zero partial isometry $w \in r(\M_{l,m}(\C) \ot P)s$ such that $\pi(x)w = w \rho(x)$ for all $x \in N$. Denote by $w_i$ the $i$-th column of $w$. Define $\cL$ as the closed linear span of $\{v_1w_iN \, | \, i = 1, \dotsc ,m\}$. Then, $\cL$ is a non-zero, since $\rE_P(w^*v_1^*v_1w) = w^*\rE_P(v_1^*v_1)w$ and $r = \supp \rE_P(v_1^*v_1)$. So $\cL$ is a non-zero $\vfi(N)$-$N$-subbimodule of $\cK$ with finite right dimension.
\end{proof}

\begin{lemma}\label{lem:contained-in-P}
Let $\cK$ be a finite index $P \vNtensor B$-$P \vNtensor B$-subbimodule of a finite index $M$-$M$-bimodule $\cH$ and let $\bim{N}{\cL}{N} \subset \bim{N}{\cK}{N}$ be an irreducible finite index $N$-$N$-subbimodule. Then $\bim{N}{\cL}{N}$ is isomorphic to a subbimodule of $\bim{N}{\Ltwo(P)}{N}$.
\end{lemma}

\begin{proof}
Assume, by contradiction,  that $\cL$ is not contained in $\bim{N}{\Ltwo(P)}{N}$. Take some non-trivial finite index irreducible $N$-$N$-bimodule $\cL^Q$ in $\Ltwo(Q)$ and some non-trivial finite index irreducible $N$-$N$-bimodule $\cL^P$ in $\Ltwo(P)$ both with right dimension greater than or equal to $1$. Denote by $\cX_0$ the $\twonorm$-closure of $\cL \cdot M$. Lemma \ref{lemm.connes_tensor} implies that $\cX_0$ is a non-zero $N$-$M$-bimodule which is isomorphic to a subbimodule of $\cH$ and lies in $\cL \otimes_N \Ltwo M$. Define the $N$-$M$-bimodules
  \[\cX_n =
    (\cL^P \otimes_N \cL^Q)^{\otimes n}
    \otimes_N \cX_0
    \, \text{.}\]
Note that $\cL \in \ANFAlg(N \subset P)$. By assumption, the fusion algebras $\FAlg(N \subset Q)$ and $\ANFAlg(N \subset P)$ are free inside $\FAlg(N)$. Therefore, as in \cite{falguieresvaes10}, the $\cX_n$ follow pairwise disjoint as $N$-$N$-bimodules and hence pairwise disjoint as $N$-$M$-bimodules.

Decompose $\cX_0 \subset \cH$ as a direct sum of irreducible $N$-$N$-bimodules $\cY_i$. Write $(\cL^Q)^0 = \cL^Q \cap Q$ and $(\cL^P)^0 = \cL^P \cap P$. Then, $(\cL^P)^0 \cdot (\cL^Q)^0 \cdots (\cL^Q)^0 \cdot \cY_i$ is non-zero. If not, we had
  \[M \cdot \cY_i \cdot M
  = M \cdot (\cL^P)^0 \cdot (\cL^Q)^0 \cdots (\cL^Q)^0 \cdot N \cdot \cY_i \cdot M
  = 0 \, \text{,}\]
contradicting the fact that $M$ is a factor. As above, the freeness assumption implies that $(\cL^P \otimes_N \cL^Q)^{\otimes n} \otimes_N \cY_i$ is irreducible. Then, by Lemma \ref{lemm.connes_tensor}, we have that $(\cL^P \otimes_N \cL^Q)^{\otimes n} \otimes_N \cY_i$ sits inside $\cH$ as the $\|\cdot \|_2$-closure of $(\cL^P)^0 \cdot (\cL^Q)^0 \cdots (\cL^Q)^0 \cdot \cY_i$.  We have proven that $\cH$ contains a copy of each $\cX_n$.

Note that $\dim_{\, \text{-} M} \left (\cX_n \right ) > \dim_{\, \text{-} M} \left ( \cX_0 \right)$. As a consequence, $\cH$, as a right $M$-module, has infinite dimension, which is a contradiction.
\end{proof}

\begin{lemma}
\label{lem:trivial-P-bimodule}
Let $\cH$ be a finite index $M$-$M$-bimodule and $\cK$ a finite index $P \vnt B$-$P \vnt B$-subbimodule. Then, $\cK$ is isomorphic to a multiple of the trivial $P$-$P$ bimodule.
\end{lemma}
\begin{proof}
By Lemma \ref{lem:fi-N-bimodule}, we have a non-zero $N$-$N$-subbimodule $\cL \subset \cK$ with finite right dimension. Then, for $x,y \in \QN_{P \vnt B}(N)$, the closure of $Nx \cdot \cL \cdot yN$ is still an $N$-$N$-subbimodule of $\cK$ with finite right dimension. Since $N \subset P \vnt B$ is quasi-regular, the linear span of all $N$-$N$-subbimodules of $\cK$ with finite right dimension is dense in $\cK$. Then, a maximality argument shows that $\cK$ can be decomposed as the direct sum $N$-$N$-subbimodules with finite right dimension. By symmetry, $\cK$ decomposes as the direct sum of $N$-$N$-bimodules with finite left dimension. As a consequence, $\cK$ may be written as the direct sum of finite index $N$-$N$-subbimodules.

Let $\cL$ be an irreducible finite index $N$-$N$-subbimodule of $\cK$. Lemma \ref{lem:contained-in-P} shows that $\cL$ is contained in $\Ltwo(P)$. Recall that
\[P = \rL_\Omega(\QQ^3 \oplus \QQ^3 \rtimes \SL(3,\QQ))\quad , \quad N = \rL_\Omega(\ZZ^3 \oplus \ZZ^3) \eqstop \]
Hence, $\cL$ arises as the $\|\cdot \|_2$-closure of $N u_g N$ for some element $g \in \QQ^3 \oplus \QQ^3 \rtimes \SL(3,\QQ)$. By almost-normality (see property $(\cP_1)$ and the remarks following it), take a finite index subgroup $\Lambda_0$ of $\ZZ^3 \oplus \ZZ^3$ such that $\Ad(g) (\Lambda_0) \subset \ZZ^3 \oplus \ZZ^3$. Denote by $\cL_0$ the closure of $Nu_g\rL_\Omega(\Lambda_0)$. Then, $\cL_0$ is an irreducible $N$-$\rL_\Omega(\Lambda_0)$-subbimodule of $\cL$. Note that
\[\cL_0 \underset{\rL_\Omega(\Lambda_0)}{\otimes} \overline{\rL_\Omega(\Lambda_0)u_g^* N}\ \cong \ \Ltwo(N) \, .\]
Lemma \ref{lemm.connes_tensor} implies that $\cK$ contains a copy of the trivial $N$-$N$-bimodule $\eL^2(N)$, realized as the $\|\cdot \|_2$-closure of $\cL_0u_g^*N$.

Write $\cK \cong \cH(\psi)$ for some finite index inclusion $\psi: P \vnt B \ra q(P \vnt B)^\infty q$, where $(\Tr \otimes \tau)(q) < \infty$. By the above paragraph, we can take a trivial $N$-$N$-bimodule inside $\cK$. Then, there is an $N$-central vector $v \in q\Ltwo(P \vnt B)^\infty$. Taking polar decomposition, we may assume that $v$ is a partial isometry in $q(\Cmatnm{\infty}{1} \vnt P \vnt B)$ satisfying $\psi(x)v = vx$, for all $x \in N$. Note that $vv^* \in \psi(N)' \cap q(P \vnt B)^\infty q$. Hence, $(q - vv^*) \cdot \cK$ defines a $N$-$P \vnt B$-subbimodule of $\cK$ and we may apply the previous procedure. As in the proof of Proposition \ref{prop:representants-for-bimodules}, a maximality argument yields a family of partial isometries $v_i \in q(\Cmatnm{\infty}{1} \vnt P \vnt B)$ satisfying $\psi(x) v_i = v_i x$, for all $x \in N$ and such that $\sum {v_iv_i^*} = q$. Putting these partial isometries in a row, we obtain an element $w \in q(P \vnt B)^\infty$ satisfying $ww^* = \sum {v_iv_i^*} = q$. By irreducibility of $N  \subset P$ (see $(\cP_1)$), we have a projection $p = w^*w \in (N' \cap P \vnt B)^\infty = B^\infty$. Conjugating $\psi$ with $w^*$ from the beginning, we obtain a finite index inclusion $\psi: P \vnt B \ra p(P \vnt B)^\infty p$, where $p \in B^\infty$ such that $(\Tr \otimes \tau)(p) < \infty$ and $\psi(x) = xp$ for all $x \in N$ and still satisfying $\cK \cong \cH(\psi)$.

Let $g \in \Gamma = \QQ^3 \oplus \QQ^3 \rtimes \SL(3,\QQ)$ and $\Lambda_0 < \ZZ^3 \oplus \ZZ^3$ be finite index subgroup such that $\Ad(g^{-1})(\Lambda_0) \subset \Lambda$. Denote by $\alpha_{g^{-1}}$ the $*$-homomorphism $\rL_\Omega(\Lambda_0) \ra \rL_\Omega(\Lambda)$ induced by $\Ad(g^{-1})$. For $x \in \rL_\Omega(\Lambda_0)$ we have
\[ \psi(u_g)u_g^*x = \psi(u_g)\alpha_{g^{-1}}(x)u_g^* = \psi(u_g)\psi(\alpha_{g^{-1}}(x))u_g^* = x \psi(u_g)u_g^* \, \text{.}\]
By $(\cP_5)$, we have that $\rL_\Omega(\Lambda_0) \subset P$ is irreducible. As a consequence, $v_g = \psi(u_g)u_g^* \in \cU(pB^\infty p)$. Note that $\psi(B) \subset \psi(N)' \cap p(P \vnt B)^\infty p = pB^\infty p$ and $v_g \in \psi(B)' \cap pB^\infty p$.

We prove that the inclusion $\psi(B) \subset pB^\infty p$ has finite index using Theorem \ref{thm.caracterisatio-bimod}. Consider the conjugate bimodule $\overline{\cK}$ of $\bim{P \vnt B}{\cK}{P \vnt B}$. As proven above, we may write $\overline{\cK} \cong \cH(\overline{\psi})$, where $\overline{\psi}:P \vnt B \ra q(P \vnt B)^\infty q$ is a finite index inclusion satisfying $\overline{\psi}(x) = x q$, for all $x \in N$ and $q \in B^\infty$ is a projection such that $(\Tr \otimes \tau)(q) < \infty$. Note that $\cK \otimes_{P \vnt B} \overline{\cK} \cong \cH((\overline{\psi} \otimes \id) \circ \psi)$. Hence there is a conjugate map $R: \Ltwo(P \vnt B) \ra \cH((\overline{\psi} \otimes \id) \circ \psi)$. Considering $R$ as an element of $(p \otimes q)(\M_{\infty,1}(\C) \vnt P \vnt B)$ we have
\[R \in  (p \otimes q)(\M_{\infty,1}(\C) \vnt P \vnt B) \cap N' = (p \otimes q)(\M_{\infty,1}(\C) \vnt B) \eqstop \]
 Define $\psi_B: B \ra pB^\infty p$ and $\overline{\psi}_B:B \ra q B^\infty q$ as the restrictions of $\psi$ and $\overline{\psi}$ to $B$ and $S: \Ltwo(B) \ra \cH((\overline{\psi}_B  \otimes \id) \circ \psi_B)$ as the restriction of $R$. Similarly, we find an intertwiner $\ol{S}: \Ltwo(B) \ra \cH((\psi_B  \otimes \id) \circ \ol{\psi_B})$, giving a pair of conjugate morphisms for $\cH(\psi_B)$. Then the same argument as in \cite[Lemma 3.2]{longoroberts97} implies that $\psi(B)' \cap pB^\infty p$ is of finite type I.

It follows that $g \mapsto v_g \in \psi(B)' \cap pB^\infty p$ is a direct integral of finite dimensional unitary representations of $\Gamma$ and hence trivial, since $\Gamma$ has no non-trivial finite dimensional unitary representations (see $(\cP_4)$). We conclude that $\psi(u_g) = u_gp$ and that $\bim{P}{\cK}{P}$ is a multiple of the trivial $P$-$P$-bimodule.
\end{proof}

\begin{lemma}\label{la:remark_trivializing_inclusions}
Let $\psi: P \vNtensor B \rightarrow p (P \vNtensor B)^n p$ be a finite index inclusion such that \[\bim{\psi(P \otimes 1)}{\bigl ( p \big (\C^n \ot \rL^2(P \vNtensor B) \big ) \bigl)}{P \otimes 1}\] is a multiple of the trivial $P$-$P$-bimodule. Then there exists a non-zero partial isometry $u \in  \M_{n,\infty}(\CC) \vnt P \vNtensor B$ such that $uu^* = p$, $q = u^*u \in B^\infty$ and $u^*\psi (x)u = qx$ for all $x \in P$, where we consider $P \subset P^\infty$ diagonally.
\end{lemma}

\begin{proof} Consider the $P$-$P$-bimodule $\cH$ given by
\[\bim{P}{\cH}{P} = \bim{\psi(P \otimes 1)}{\bigl ( p \big (\C^n \ot \rL^2(P \vNtensor B) \big ) \bigl )}{P \otimes 1} \eqstop\]
 Since $\cH$ is a multiple of the trivial $P$-$P$-bimodule, there exists a non-zero vector $v \in p \big (\C^n \otimes \Ltwo(P \vNtensor B) \big )$ such that $\psi(x)v = v x$ for all $x \in P$. Taking its polar decomposition, we may assume that $v$ is a non-zero partial isometry in $p(\C^n \otimes P \vNtensor B)$. As in the proof of Proposition \ref{prop:representants-for-bimodules}, a maximality argument provides a family of non-zero partial isometries $(v_i)$, inside $p(\C^n \otimes P \vNtensor B )$, satisfying $\psi(x)v_i = v_i x$ for all $x \in P$ and such that $p = \sum v_i v_i^*$. Putting all $v_i$ in one row, we get $u \in \M_{n,\infty}(\CC) \vnt P \vNtensor B$. Then $\psi(x)u = ux$, for all $x \in P$. We also have that $uu^* = \sum v_iv_i^* = p$ and $u^*u \in (1 \otimes P \otimes 1)' \cap (P \vNtensor B)^\infty = B^\infty$. Thus, $u$ is the required partial isometry.
\end{proof}

\begin{proof}[Proof of Theorem \ref{thm.main}]
Let $\bim{M}{\cH}{M}$ be a finite index irreducible $M$-$M$-bimodule. We prove that $\cH$ is isomorphic to a bimodule in the range of the functor $F : \bimod(Q \subset Q_1) \to \bimod(M)$, constructed in Section \ref{sec:functor}. We do this in two steps.

\textit{Step 1.} There exists a projection $p \in (P \vNtensor B)^\infty$ with $(\Tr \ot \tau)(p) < +\infty$ and $*$-homomorphism $\psi:M \rightarrow p M^\infty p$ such that
\begin{itemize}
\item $\psi(M) \subset p M^\infty p$ has finite index,
\item $\psi(P \vNtensor B) \subset p(P \vNtensor B)^\infty p$ and this inclusion has essentially finite index and
\item $\bim{M}{\cH}{M} \cong \bim{M}{\cH(\psi)}{M}$.
\end{itemize}

\emph{Proof of Step 1.} Let $\psi:M \rightarrow p M^n p$ be a finite index inclusion such that $\bim{M}{\cH}{M} \cong \bim{M}{\cH(\psi)}{M}$. By symmetry, Theorem \ref{thm.intertw.vaes} and the remarks preceding it, we are left with proving the two following statements.
\begin{enumerate}
\item $\psi(P \vNtensor B)q \prec_{M} P \vNtensor B$, for every projection $q \in \psi(P \vNtensor B)' \cap p M^n p$.
\item Whenever $\cK \subset \rL^2(M)$ is a $(P \vNtensor B)$-$(P \vNtensor B)$-subbimodule with $\dim_{\, \text{-} {P \vnt B}}(\cK) < + \infty$, we have $\cK \subset \Ltwo(P \vNtensor B)$.
\end{enumerate}

By assumption, there is no non-trivial $*$-homomorphism from $N_0$ to any amplification of $Q$. It follows that $\psi(N_0) \not \prec_M Q$. Hence, Theorem \ref{IPP1} implies that $\psi(N_0) \prec_M P \vNtensor B$. So there is a $*$-homomorphism $\varphi:N_0 \rightarrow q(P \vNtensor B)^mq$ and a non-zero partial isometry $v \in p(\M_{n,m}(\C) \ot M)q$ such that $\psi(x)v = v \varphi(x)$ for all $x \in N_0$. We have $v^*v \in \varphi(N_0)' \cap qM^mq$. So $v^*v \in q(P \vNtensor B)^m q$ by Theorem \ref{IPP2}. Then,
\[v^*v (\QN_{qM^mq}(\varphi(N_0))'') v^*v \subset q (P \vNtensor B)^m q \eqcomma\]
 by Theorem \ref{IPP2} again. Since $N_0 \subset P$ is quasi-regular (see $(\cP_2)$), we also have that
\[v^*\psi(P \vNtensor B)v \subset v^*v (\QN_{qM^mq}(\varphi(N_0))'') v^*v \subset q(P \vnt B)^mq \eqstop \]
Note that all the previous arguments remain true when cutting down $\psi$ with a projection in $\psi(P \vNtensor B)' \cap p M^n p$, so we have proven (i). Theorem \ref{IPP2} implies (ii) and Step 1 is proven.

\textit{Step 2.} We may assume that $p \in B^\infty$ and that the $*$-homomorphism $\psi$ satisfies
\begin{itemize}
\item
$\psi(x) = px$, for all $x \in P$,
\item
$\psi(B) \subset B^\infty$,
\item
$\psi(Q) \subset p Q^\infty p$.
\end{itemize}

\emph{Proof of Step 2.} By Step 1, the inclusion $\psi(P \vnt B) \subset p (P \vnt B)^\infty p$ has essentially finite index. Let $q$ be a projection in $\psi(P \vnt B)' \cap p(P \vnt B)^\infty p$ such that $\cK = \bim{\psi(P \vnt B)}{\bigl (q \Ltwo(P \vNtensor B)^\infty \bigr )}{P \vnt B}$ is a finite index $P \vnt B$-$P \vnt B$-subbimodule of $\bim{P \vnt B}{\cH}{P \vnt B}$. Lemma \ref{lem:trivial-P-bimodule} implies that $\bim{P \ot 1}{\cK}{P \ot 1}$ is isomorphic to a multiple of the trivial $P$-$P$-bimodule. Lemma \ref{la:remark_trivializing_inclusions} yields a non-zero partial isometry $u \in q ( \M_{\infty, m}(\C) \ot P \vNtensor B)$ satisfying $u^*\psi(x)u = u^*u x$ for all $x \in P$ and such that $uu^* = q$ and $u^*u \in B^m$. Since $\psi(P \vNtensor B) \subset p(P \vNtensor B)^\infty p$ has essentially finite index, this procedure provides a non-zero partial isometry $u \in (P \vNtensor B)^\infty$ satisfying $u^*\psi(x)u = u^*u x$ for all $x \in P$ with $uu^* = p$ and $u^*u \in B^\infty$. Conjugating $\psi$ with $u^*$ from the beginning, we may assume that $p \in B^\infty$ and $\psi(x) = p x$ for all $x \in P$.

We have $P' \cap M = B$ and $\psi(x) = p x$ for all $x \in P$, with $p \in B^\infty$, therefore $\psi(B) \subset B^\infty$.

Since $p \in (N' \cap Q)^\infty$ and $\psi(x) =  p x$ for all $x \in P$, the $*$-homomorphism $\psi$ extends to an $N$-$N$-bimodular map $v: \rL^2(M) \to \rL^2(p M^\infty p)$. By freeness of $\FAlg(N \subset Q)$ and $\FAlg(N \subset P)$ inside $\FAlg(N)$, we have that $v(\rL^2(Q))$ is an $N$-$N$-subbimodule of $\Ltwo(p Q^\infty p)$. Hence $\psi(Q) \subset p Q^\infty p$, which ends the proof of Step $2$.
\end{proof}

\subsection{Proof of Theorem \ref{thm:construct-bimodule-categories}}
\label{proof.construct-bimodule-categories}

We use the following version of \cite[Theorem 0.2]{nicoarapopasasyk07} for the proof of Theorem \ref{thm:construct-bimodule-categories}.

\begin{theorem}[See {\cite[Theorem 0.2]{nicoarapopasasyk07}}]\label{thm.nicoarapopasasyk}
Let $\Gamma$ be a property (T) group and $M$ a separable II$_1$ factor. Let $J \subset \mathrm{H}^2(\Gamma, \mathrm{S}^1)$ be the set of scalar 2-cocycles $\Omega$ such that there exists a (not necessarily unital) non-trivial $*$-homomorphism from $\rL_\Omega(\Gamma)$ to an amplification of $M$. Then $J$ is countable.
\end{theorem}

\begin{proof}[Proof of Theorem \ref{thm:construct-bimodule-categories}]
Fix an inclusion of II$_1$ factors $N \subset Q$ and assume that $N$ is hyperfinite. Suppose also that $N \subset Q$ is quasi-regular and has depth 2. Denote by $N \subset Q \subset Q_1$ the basic construction.

Let $\alpha \in \RR-\QQ$ and consider the groups $\Lambda, \Gamma$ and the scalar $2$-cocycle $\Omega_\alpha \in \rZ^2(\Gamma, S^1)$ defined in at the beginning of Section \ref{section.thmC}. Since the group $\ZZ^3 \oplus \ZZ^3 \rtimes \SL(3,\ZZ)$ has property (T), Theorem \ref{thm.nicoarapopasasyk} implies that there are uncountably many $\alpha \in \RR-\QQ$ such that there is no non-trivial $*$-homomorphism from $N_0 = \rL_{\Omega_{\alpha}}\bigl ( \ZZ^3 \oplus \ZZ^3 \rtimes \SL(3,\ZZ) \bigl )$ to any amplification of $Q$. Take one such $\alpha \in \RR - \QQ$. Note that by $(\cP_1)$, $(\cP_3)$ and Lemma \ref{lem.countablefusionalgebra}, the fusion algebra $\cF = \ANFAlg \bigl( \rL_{\Omega_\alpha}(\Lambda) \subset \rL_{\Omega_\alpha}(\Gamma) \bigl)$ is countable.

Observe that $\rL_{\Om_\al}(\Lambda)$ and $N$ are two copies of the hyperfinite II$_1$ factor and take an isomorphism $\theta : N \to \rL_{\Om_\al}(\Lambda)$. Then, the fusion algebra $\cF^\theta$ may be viewed as a fusion subalgebra of $\FAlg(N)$. Since $\FAlg(N \subset Q)$ is a countable fusion subalgebra of $\FAlg(N)$, Theorem \ref{theo.free} allows us to choose $\theta$ such that $\cF^\theta$ is free with respect to $\FAlg(N \subset Q)$. We identify $N$ and $\rL_{\Om_\al}(\Lambda)$ through this isomorphism and all assumptions of Theorem \ref{thm.main} are satisfied. Write $P_\alpha = \rL_{\Omega_\alpha}(\Gamma)$ and write
\[ M_\al = (P_\al \vnt B) \underset{N \vnt B}{*} Q \quad , \quad \text{where} \quad B = N' \cap Q \eqstop\]
Using Theorem \ref{thm.main}, we obtain that $\bimod(M_\alpha) \simeq \bimod(Q \subset Q_1)$.

We prove that stable isomorphism classes of $M_\alpha$, $\alpha \in \RR - \QQ$, are countable. Assume by contradiction that there exists an uncountable subset $J \subset \RR - \QQ$ such that $M_{\alpha_j}$ are pairwise stably isomorphic, for $j \in J$. We find $k \in J$ and an uncountable subset $I \subset J$ such that $M_{\alpha_i}$ embeds (not necessarily unitally) into $M_{\alpha_k}$, for all $i \in I$. In particular, $\rL_{\Omega_{\alpha_i}}\bigl ( \ZZ^3 \oplus \ZZ^3 \rtimes \SL(3,\ZZ) \bigl )$ embeds into $M_{\alpha_k}$, for all $i \in I$. Since $\ZZ^3 \oplus \ZZ^3 \rtimes \SL(3,\ZZ)$ is a property (T) group and cohomology classes of the cocycles
\[ \bigl ( \Omega_\alpha \bigl)|_{\ZZ^3 \oplus \ZZ^3 \rtimes \SL(3,\ZZ)} \quad , \quad \alpha \in \R-\Q\]
are two by two non-equal, this contradicts Theorem \ref{thm.nicoarapopasasyk}.
\end{proof}

%%% Local Variables: 
%%% mode: latex
%%% TeX-master: "construction-bimodule-categories-hyperfinite-case"
%%% End: 

%% file: applications.tex
\section{Applications}
\label{sec:applications}

\subsection{Examples of categories that arrise as $\bimod(M)$}
\label{sec:examples}

In this part, we give examples of categories that arise as $\Bimod(M)$ of some II$_1$ factor $M$. Note that the results in \cite{vaes09} and in \cite{falguieresvaes10} show that the trivial tensor $\Cstar$-category and the category of finite dimensional representation of every compact, second countable group can be realized as a category of bimodules.

\subsubsection{Finite tensor \Cstar-categories}
\label{sec:finite-categories}

The following reconstruction theorem for finite tensor \Cstar-categories is well known, but for convenience, we give a short proof. We use Jones' planar algebras \cite{jones99} and Popa's reconstruction theorem for finite depth standard invariants \cite{popa90}. See also \cite{bigelowmorrisonpeterssnyder09} for a similar statement.

\begin{theorem}
\label{thm:reconstruct}
Let $\cC$ be a finite tensor \Cstar-category. Then there exists a finite index depth $2$ inclusion $Q \subset Q_1$ of hyperfinite II$_1$ factors such that $\bimod(Q \subset Q_1) \simeq \cC$.
\end{theorem}
\begin{proof}
  We define a depth $2$ subfactor planar algebra $P$, such that the inclusion of hyperfinite II$_1$ factors $Q \subset Q_1$ associated with it by \cite{popa90, jones99} satisfies $\Bimod(Q \subset Q_1) \simeq \cC$. Let $x \in \cC$ be the direct sum of representatives for every isomorphism class of irreducible objects in $\cC$. Denote by $\overline{x}$ the conjugate object of $x$. Let \[P_k := \mathrm{End}(\underbrace{x \otimes \overline{x} \otimes \dotsm \otimes x}_{k \text{ factors}}) \; \text{.}\] We prove that $P = \bigcup P_k$ is a subfactor planar algebra.
  Composition of endomorphisms and the $*$-functor of $\cC$ make $P$ a $*$-algebra. The categorical trace of $\cC$ defines a positive trace on $P$. Moreover, the graphical calculus for tensor \Cstar-categories induces an action of the planar operad on $P$. We have $\dim P_0 = 1$, since $1_\cC$ is irreducible. Moreover, for all $k$ we have $\dim P_k < \infty$, since $\cC$ is finite. Finally, the closed loops represent the number $\dim_\cC x \neq 0$. So $P$ is a subfactor planar algebra. It has depth $2$, since $\dim \cZ(P_k)$ is the number of isomorphism classes of irreducible objects of $\cC$ for every $k \geq 1$ and, in particular, $\dim \cZ(P_1) = \dim \cZ(P_3)$.

  Note that, in the language of \cite{popa90}, finite depth subfactor planar algebras correspond to canonical commuting squares \cite{bisch97, jones08}. So, by \cite{popa90}, there is an inclusion $Q \subset Q_1$ of hyperfinite II$_1$ factors with associated planar algebra $P^{Q \subset Q_1} \cong P$. Then $x$ corresponds to $\bim{Q}{\Ltwo(Q_1)}{Q_1}$. Let  $\cD = \bimod(Q \subset Q_1)$ and denote by $Q \subset Q_1 \subset Q_2$ the basic construction. If $p,q$ are minimal projections in $Q' \cap Q_2$, we canonically identify $\Hom_{Q\text{-}Q}(p\Ltwo(Q_1), q\Ltwo(Q_1))$ with $q (Q' \cap Q_2)p$. This defines a $\Cstar$-functor $F: \cD \rightarrow \cC$ sending $p \Ltwo(Q_1) \cong p(\bim{Q}{\Ltwo(Q_1) \otimes_{Q_1} \Ltwo(Q_1)}{Q})$ to $p(x \otimes \overline{x})$ and mapping morphisms as given by the identification $P^{Q \subset Q_1} \cong P$. Then $F$ is fully faithful and essentially surjective. We have to prove that $F$ preserves tensor products. Let $p$,$q$ be projections in $Q' \cap Q_2$. The shift-by-two operator $\mathrm{sh}_2: \: P_2 \rightarrow P_4$ is defined by adding two strings on the left. By \cite{bisch97}, we have $p\Ltwo(Q_1) \otimes_Q q\Ltwo(Q_1) \cong p \cdot \mathrm{sh}_2(q)\Ltwo(Q_2)$ as $Q$-$Q$-bimodules. On the other hand, we have $p(x \otimes \overline{x}) \otimes q(x \otimes  \overline{x}) \cong (p \otimes q)(x \otimes \overline{x} \otimes x \otimes \overline{x})$ in $\cC$. Since under the identification $P_k \cong Q' \cap Q_k$ the shift-by-two operator corresponds to $q \mapsto 1 \otimes q$, we have $\cC \simeq \cD$ as tensor $\Cstar$-categories. This completes the proof.
\end{proof}

\begin{proof}[Proof of Theorem \ref{thm:finite-tensor-categories}]
  Let $\cC$ be a finite tensor $\Cstar$-category. By Theorem \ref{thm:construct-bimodule-categories} it suffices to show that there is a finite index, depth 2 inclusion $N \subset Q$ of hyperfinite II$_1$ factors, such that for the basic construction $N \subset Q \subset Q_1$ we have $\bimod(Q \subset Q_1) \simeq \cC$. Indeed, if $N \subset Q$ is of finite index, then it is quasi-regular. By Theorem \ref{thm:reconstruct}, there is a finite index depth 2 inclusion $N_{-1} \subset N$ of hyperfinite II$_1$ factors such that $\bimod(N_{-1} \subset N) \simeq \cC$. Let $N_{-1} \subset N \subset Q \subset Q_1$ be the basic construction. Then $N \subset Q$ is a finite index, depth 2 inclusion and $\bimod(Q \subset Q_1) \simeq \bimod(N_{-1} \subset N) \simeq \cC$.
\end{proof}

\subsubsection{Representation categories}
\label{sec:representation-categories}

In \cite{falguieresvaes10} the categories of finite dimensional representations of compact second countable groups were realized as bimodule categories of a II$_1$ factor. As already mentioned, this forms a natural class of tensor $\Cstar$-categories, since they can be abstractly characterized as symmetric tensor $\Cstar$-categories with at most countably many isomorphism classes of irreducible objects. We realize categories of finite dimensional representations of discrete countable groups and of finite dimensional corepresentations of certain discrete Kac algebras as bimodule categories of a II$_1$ factor.  Neither does this class of categories have an abstract characterization, nor does the finite dimensional corepresentation theory of a discrete Kac algebra describe it completely.  However, Corollary \ref{cor:uncountable-bimodule-categories} shows that we have interesting applications coming from this class of tensor $\Cstar$-categories.

For notation concerning quantum groups, we refer the reader to the appendix in Section \ref{sec:appendix}. 

\begin{definition}[See Section 4.5 and Theorem 4.5 of \cite{soltan05}]
A discrete Kac algebra $A$ is called maximally almost periodic, if there is a family of finite dimensional corepresentations $U_n \in A \otimes \rB(H_{U_n})$ such that $A = \cspan\{ (\id \otimes \omega)(U_n) \amid n\in \NN, \omega \in \rB(H_{U_n})_*\}$  
\end{definition}

\begin{theorem}
\label{thm:reconstruct-categories-kac}
  Let $A$ be a discrete Kac algebra admitting a strictly outer action on the hyperfinite II$_1$ factor. Then there is a II$_1$ factor $M$ such that $\bimod(M) \simeq \CoRepfin(A^{\mathrm{coop}})$.
\end{theorem}
\begin{proof}
  Since $A$ acts strictly outerly on the hyperfinite II$_1$ factor $R$, the inclusion $R \subset A \ltimes R \subset \wh{A}^{\mathrm{coop}} \ltimes A \ltimes R$ is a basic construction by \cite[Proposition 2.5 and Corollary 5.6]{vaes01}. The inclusion $R \subset A \ltimes R$ has depth $2$ by \cite[Corollary 5.10]{vaes01} and since $A$ is discrete, it is quasi-regular. Moreover, we have $\bimod(A \ltimes R \subset \wh{A} \ltimes A \ltimes R) \simeq \CoRepfin(A^{\mathrm{coop}})$ by Theorem \ref{thm:representation-categories-for-subfactors}. So Theorem \ref{thm:construct-bimodule-categories} yields a II$_1$ factor $M$ such that $\bimod(M) \simeq \CoRepfin(A^{\mathrm{coop}})$.
\end{proof}

\begin{proof}[Proof of Theorems \ref{thm:representation-categories}]
  By Theorem \ref{thm:reconstruct-categories-kac} it suffices to show that every discrete group $G$ and every amenable and every maximally almost periodic Kac algebra $A$ has a strictly outer action on the hyperfinite II$_1$ factor $R$. 

Let us first consider the case of a discrete group. The non-commutative Bernoulli shift $G \grpaction{} (\Cmat{2}, \tr)^{\otimes G}$ is well known to be outer. It is clear that $\otimes_{n=1}^\infty (\Cmat{2}, \tr)$ is isomorphic to $R$.

First note that $(A^{\mathrm{coop}})^{\mathrm{coop}} = A$ for all quantum groups $A$. By Vaes \cite[Theorem 8.2]{vaes05-strictly-outer-actions}, it suffices to show that every amenable and every maximally almost periodic Kac algebra $A$ has a faithful corepresentation of $A^{\mathrm{coop}}$ in the hyperfinite II$_1$ factor.

If $A$ is a discrete amenable Kac algebra, then so is $A^{\mathrm{coop}}$. By \cite[Proposition 8.1]{vaes05-strictly-outer-actions}, $A^{\mathrm{coop}}$ has a faithful corepresentation into $R$.  If $A$ is a discrete maximally almost periodic Kac algebra, then $A^{\mathrm{coop}}$ is also maximally almost periodic, since $A$ has a bounded antipode. There is a countable family of corepresentations $U_n$ of $A^{\mathrm{coop}}$ whose coefficients span $A$ densely.  Considering $\oplus_n \rB(H_{U_n})$ as a von Neumann subalgebra of $R$, the corepresentation $\boxplus_n U_n$ of $A^{\mathrm{coop}}$ is faithful.
\end{proof}

As a corollary of Theorem \ref{thm:representation-categories}, we get the following improvement of \cite[Corollary 8.8]{ioanapetersonpopa08} and \cite{falguieresvaes08}. This is the first example of an explicitly known bimodule category with uncountably many isomorphism classes of irreducible objects.

\begin{corollary}
\label{cor:uncountable-bimodule-categories}
  Let $G$ be a second countable, compact group. Then there is a II$_1$ factor $M$ such that $\Out(M) \cong G$ and every finite index bimodule of $M$ is of the form $\cH(\alpha)$ for some $\alpha \in \Aut(M)$. In particular, the bimodule category of $M$ can be explicitly calculated and has an uncountable number of isomorphism classes of irreducible objects.
\end{corollary}

The exact sequence $1 \ra \Out(M) \ra \mathrm{grp}(M) \ra \cF(M) \ra 1$ shows that the fundamental group of $M$ obtained in Corollary \ref{cor:uncountable-bimodule-categories} is trivial.  Note, that the factors constructed in \cite{falguieresvaes08, ioanapetersonpopa08} also have trivial fundamental group.

\begin{proof}
 Let $G$ be a second countable, compact group.  By \cite[Theorem 4.2]{soltan06}, $\rL(G)$ is maximally almost periodic and its irreducible, finite dimensional corepresentations are one dimensional and indexed by elements of $G$. Their tensor product is given by multiplication in $G$. So we can apply Theorem \ref{thm:representation-categories} to the discrete Kac algebra $\rL(G)$ in order to obtain $M$.
\end{proof}

\subsection{Possible indices of irreducible subfactors}
\label{sec:indices}

In this section, we investigate the structure of subfactors of the II$_1$ factor $M$ that we obtained in Theorem \ref{thm:finite-tensor-categories}.  We write
\[\cC(M) = \left \{ \lambda \mid \text{there is an irreducible finite index subfactor of } M \text{ with index } \lambda \right \} \, \text{.}\]
We use the fact that the lattice of irreducible subfactors of a II$_1$ factor is actually encoded in its bimodule category. In special cases, indices of irreducible subfactors correspond to Frobenius-Perron dimensions (see \cite[Section 8]{etingofnikshychostrik05}) of objects in the bimodule category. Using recent work on tensor categories \cite{grossmansnyder11} and Theorem \ref{thm:finite-tensor-categories}, we give examples of of II$_1$ factors  $M$ such that $\cC(M)$ can be computed explicitly and contains irrationals.

\begin{definition}[See \cite{fuchsschweigert03, yamagami04}]
Let $\cC$ be a compact tensor $\Cstar$-category with tensor unit $1_\cC$.
\begin{enumerate}
  \item An algebra $(\cA,m,\eta)$ in $\cC$ is an object $\cA$ in $\cC$ with multiplication and unit maps $m: \cA \otimes \cA \rightarrow \cA$ and $\eta: 1_\cC \rightarrow \cA$ such that the following diagrams commute

  \centerline{
  \begin{xy}
    \xymatrix{\cA \otimes \cA \otimes \cA \ar[r]^-{m \otimes \id} \ar[d]^{\id \otimes m}	& \cA \otimes \cA \ar[d]^m	\\
	      \cA \otimes \cA \ar[r]^-{m}							& \cA}
  \end{xy}
  \phantom{xxxx}
  \begin{xy}
    \xymatrix{\cA \otimes 1_\cC \ar[d]^{\id \otimes \eta}	& \cA \ar[r]^-{\cong} \ar[l]_-{\cong} \ar@{=}[d]	& 1_\cC \otimes \cA \ar[d]^{\eta \otimes \id} \\
	      \cA \otimes \cA \ar[r]^-{m}				& \cA						& \cA \otimes \cA \, \text{.} \ar[l]_-{m}} 
  \end{xy}}

  \item A coalgebra $(\cA, \Delta, \epsilon)$ in $\cC$ is an object $\cA$ in $\cC$ with comultiplication and counit map $\Delta: \cA \rightarrow \cA \otimes \cA$ and $\epsilon: \cA \rightarrow 1_\cC$ such that $(\cA, \Delta^*, \epsilon^*)$ is an algebra.
  \item A Frobenius algebra $(\cA, m, \eta, \Delta, \epsilon)$ in $\cC$ is an object $\cA$ in $\cC$ with maps $m$, $\eta$, $\Delta$, $\epsilon$ such that $(\cA, m, \eta)$ is an algebra, $\Delta = m^*$, $\epsilon = \eta^*$ and
  \[(\id \otimes m) \circ (\Delta \otimes \id) = \Delta \circ m = (m \otimes \id) \circ (\id \otimes \Delta) \, \text{.}\]
  \item A Frobenius algebra $(\cA, m, \eta, \Delta, \epsilon)$ is special if $\Delta$ and $\eta$ are isometric.
  \item A Frobenius algebra $\cA$ is irreducible if $\dim (\Hom(1_\cC, \cA)) = 1$.
\end{enumerate}
\end{definition}

\begin{remark}
Note that the notion of a special Frobenius algebra is equivalent to the notion of a Q-system \cite{longoroberts97}.
\end{remark}

The following lemma and proposition are probably well known, but since we could not find a reference, we give a short proof for convenience of the reader.

\begin{lemma}\label{la:lem:inclusionsgivefrobenius}
  Let $M \subset M_1$ be a finite index inclusion of tracial von Neumann algebras. Then $\Ltwo(M_1)$ is a special Frobenius algebra in $\bimod(M)$. The Frobenius algebra $\Ltwo(M_1)$ is irreducible if and only if $M \subset M_1$ is irreducible.
\end{lemma}
\begin{proof}
  We prove that $\Ltwo(M_1)$ is an algebra in $\bimod(M)$ with coisometric multiplication and isometric unit. By \cite{longoroberts97}, this shows that $\Ltwo(M_1)$ is a special Frobenius algebra. The multiplication on $\Ltwo(M_1)$ is given by $m(x \otimes_M y) = xy$, for $x, y \in M$. The commutative diagram
  
  \centerline{
  \begin{xy}
  \xymatrix{\Ltwo(M_1) \otimes_M \Ltwo(M_1) \ar[r]^-{m} \ar[d]^{\cong} &
	      \Ltwo(M_1) \\
	    \Ltwo(M_2) \ar[ur]^{e}
  }
  \end{xy}
  }
  
  proves that $m$ is well defined and coisometric. Here we denote by $M \subset M_1 \subset M_2$ the basic construction and we denote by $e$ the Jones projection. The unit map of $\Ltwo(M_1)$ is given by the canonical embedding $\Ltwo(M) \rightarrow \Ltwo(M_1)$.
  
  The inclusion $M \subset M_1$ is irreducible if and only if $\bim{M}{\Ltwo(M_1)}{M_1}$ is irreducible if and only if $\bim{M}{\Ltwo(M_1)}{M} \cong \bim{M}{\Ltwo(M_1) \otimes_{M_1} \Ltwo(M_1)}{M}$ contains a unique copy of $\bim{M}{\Ltwo(M)}{M}$.
\end{proof}

Whenever $\cH$ is a finite index $M$-$M$-bimodule over a II$_1$ factor $M$, we denote by $\cH^0$ the set of bounded vectors in $\cH$. Recall that $\cH^0$ is dense in $\cH$.

\begin{proposition}\label{la:prop:characterisationinclusions}
Let $M$ be a II$_1$ factor. Then there is a bijection between irreducible special Frobenius algebras in $\bimod(M)$ and irreducible finite index inclusions $M \subset M_1$ of II$_1$ factors. The bijection is given by
\[ \cH \mapsto (M \subset \cH^0) \text{ and } (M \subset M_1) \mapsto \bim{M}{\Ltwo(M_1)}{M} \, \text{.}\]
\end{proposition}
\begin{proof}
Lemma \ref{la:lem:inclusionsgivefrobenius} shows that $\Ltwo(M_1)$ is an irreducible special Frobenius algebra for all irreducible finite index inclusions $M \subset M_1$.  Let $(\cH, m, \epsilon, \Delta, \eta)$ be an irreducible special Frobenius algebra in $\bimod(M)$. We have to prove that $M \subset \cH^0$ is a finite index, irreducible inclusion of von Neumann algebras. Let $M_2 = \Hom_{\text{-}M}(\cH)$ be the commutant of the right $M$-action. Then

\centerline{
\begin{xy}
  \xymatrix{\cH^0 \otimes \cH \ar[r] \ar[dr] &
	      \cH \otimes_M \cH \ar[d]^{m} \\
	    & 
	      \cH
  }
\end{xy}
}

yields a map $\phi: \cH^0 \rightarrow M_2$. By considering the restriction  $m:\cH^0 \otimes \Ltwo(M) \rightarrow \cH$, it is clear that $\phi$ is injective. Consider the special Frobenius algebra $(\Ltwo(M_2), m_2, \eta_2, \Delta_2, \epsilon_2)$. We prove that $\phi(\cH^0)$ is a Frobenius subalgebra of $\Ltwo(M_2)$. Indeed, the composition $\cH^0 \otimes \cH^0 \rightarrow \cH \otimes_M \cH \rightarrow \cH$ induces a multiplication on $\cH^0$, since $M$-$M$-bimodular maps send $M$-bounded vectors to $M$-bounded vectors. Since $m$ is associative, we have for $\xi, \xi' \in \cH^0$
\[\phi(m(\xi, \xi')) \cdot \widehat{1_M} = m \circ (m \otimes \id) (\xi, \xi', 1_M) = m \circ (\id \otimes m)(\xi, \xi', 1_M) = \phi(\xi) \cdot \widehat{\phi(\xi')} = \widehat{\phi(\xi) \cdot \phi(\xi')} \, \text{.}\]
So $m$ is the restriction of $m_2$. By taking adjoints, we see that $\Delta$ is the restriction of $\Delta_2$. Next, note that $m \circ (\eta \otimes \id) = \id = m_2 \circ (\eta_2 \otimes \id)$. So $\phi(\eta(x)) \cdot \xi = x\xi$ for all $x \in M \subset \Ltwo(M)$ and all $\xi \in \cH$. So $\eta$ agrees with $\eta_2$. Again, by taking adjoints, $\epsilon$ is the restriction of $\epsilon_2$.

Let $R: \Ltwo(M) \rightarrow \overline{\cH} \otimes_M \cH$ denote the standard conjugate for $\cH$ \cite{longoroberts97}. Frobenius algebras are self-dual via $\Delta \circ \eta$, that is $\Delta \circ \eta: \Ltwo(M) \rightarrow \cH \otimes_M \cH$ is a conjugate for $\cH$. In particular, there is an $M$-$M$-bimodular isomorphism $\psi:\cH \rightarrow \overline{\cH}$ such that

\centerline{
\begin{xy}
\xymatrix{
  \overline{\cH} \otimes_M \cH \ar[r]^-{R^*}			& \Ltwo(M) \\
  \cH \otimes_M \cH \ar[r]^-{m} \ar[u]_{\psi \otimes \id}	& \cH \ar[u]_\epsilon
} 
\end{xy}}

commutes. Denoting by $R_2: \Ltwo(M) \rightarrow \overline{\Ltwo(M_2)} \otimes_M \Ltwo(M_2)$ the standard conjugate for $\Ltwo(M_2)$, we have the commuting diagram

\centerline{
\begin{xy}
\xymatrix{
  \overline{\Ltwo(M_2)} \otimes_M \Ltwo(M_2) \ar[r]^-{R_2^*}					& \Ltwo(M) \, \phantom{\text{.}} \\
  \Ltwo(M_2) \otimes_M \Ltwo(M_2) \ar[r]^-{m_2} \ar[u]_{(\overline{\phantom{x}} \circ *) \otimes \id}	& \Ltwo(M_2) \, \text{.} \ar[u]_{\epsilon_2}
} 
\end{xy}}

Note that, by the definition of standard conjugates, $R$ is the composition of $R_2$ with the orthogonal projection $\overline{\Ltwo(M_2)} \otimes_M \Ltwo(M_2) \rightarrow \overline{\cH} \otimes_M \cH$. So $\psi$ is the restriction of $\overline{\phantom{x}} \circ *$. Now consider the commutative diagram

\centerline{\begin{xy}
\xymatrix{
  \overline{\Ltwo(M_2)} \otimes_M \Ltwo(M_2) \ar[r]^-{\id \otimes \Delta_2}	
    & \overline{\Ltwo(M_2)} \otimes_M \Ltwo(M_2) \otimes_M \Ltwo(M_2) \ar[r]^-{R_2^* \otimes \id}
    & \Ltwo(M) \otimes_M \Ltwo(M_2) \\
  \Ltwo(M_2) \otimes_M \Ltwo(M_2) \ar[r]^-{\id \otimes \Delta_2} \ar[u]_{(\overline{\phantom{x}} \circ *) \otimes \id}
    & \Ltwo(M_2) \otimes_M \Ltwo(M_2) \otimes_M \Ltwo(M_2) \ar[r]^-{m_2 \otimes \id}  
    & \Ltwo(M_2) \otimes_M \Ltwo(M_2) \ar[u]^{\epsilon_2 \otimes \id} \, \text{.}
} 
\end{xy}}

It restricts to the corresponding diagram with $\Ltwo(M_2)$ replaced by $\cH$. Define $\overline{m} = (R^* \otimes \id) \circ (\id \otimes \Delta): \overline{\cH}^0 \otimes \cH \rightarrow \cH$ and $\overline{m_2} = (R_2^* \otimes \id) \circ (\id \otimes \Delta_2)$. Since
\[m_2 = (\epsilon_2 \otimes \id) \circ (m_2 \otimes \id) \circ (\id \otimes \Delta_2)\]
in the Frobenius algebra $\Ltwo(M_2)$, we have that

\centerline{
\begin{xy}
\xymatrix{
  \overline{M_2} \otimes \Ltwo(M_2) \ar[r]^-{\overline{m_2}}		& \Ltwo(M_2) \\
  M_2 \otimes \Ltwo(M_2) \ar[r]^-{m_2} \ar[u]_{(\overline{\phantom{x}} \circ *) \otimes \id}	& \Ltwo(M_2) \ar@{=}[u]
} 
\end{xy}
\qquad \raisebox{-20pt}{\text{and}} \qquad
\begin{xy}
\xymatrix{
  \overline{\cH}^0 \otimes \cH \ar[r]^-{\overline{m}}				& \cH \\
  \cH^0 \otimes \cH \ar[r]^-{m} \ar[u]_{\psi \otimes \id}	& \cH \ar@{=}[u]  \, \text{.}
} 
\end{xy}}

commute and the second diagram is a restriction of the first one. Denote by $\overline{\phi}: \overline{\cH}^0 \rightarrow M_2$ the embedding defined by $\overline{m}$. Then $\overline{\phi}(\overline{x}) = \phi(x)^*$ for $x \in \cH^0$ and $\phi(\cH^0) = \overline{\phi}(\overline{\cH}^0)$. This proves that $\phi(\cH^0)$ is closed under taking adjoints.

We already proved that $\phi(\cH^0)$ is a $*$-subalgebra of $M_2$. Since $\lmo{M}{\cH}$ has finite dimension, $\phi(\cH^0)$is finitely generated over $M$. Hence, it is weakly closed in $M_2$, so it is a von Neumann subalgebra. Finally, $\bim{M}{\Ltwo(\cH^0)}{M} \cong \bim{M}{\cH}{M}$, so $M \subset \cH^0$ is irreducible and has finite index.
\end{proof}

\begin{remark}\label{la:rem:characterisationindices}
\begin{enumerate}
  \item By uniqueness of multiplicative dimension functions on finite tensor $\Cstar$-categories, see \cite{calaqueetingof08}, we have $[M_1 : M]_\mathrm{min} = \mathrm{FPdim}(\bim{M}{\Ltwo(M_1)}{M})$, where $[M_1:M]_\mathrm{min}$ denotes the minimal index \cite{havet90, longo92} and $\mathrm{FPdim}$ denotes the Frobenius Perron dimension \cite[Section 8]{etingofnikshychostrik05}. So if $M \subset M_1$ is extremal (for example irreducible), then we have $[M_1:M] = \mathrm{FPdim}(\bim{M}{\Ltwo(M_1)}{M})$.
  \item By Proposition \ref{la:prop:characterisationinclusions}, irreducible special Frobenius algebras correspond to irreducible inclusions $M \subset M_1$, hence to irreducible subfactors $N \subset M$. In particular, if $\bimod(M)$ is finite, then $\cC(M) = \left \{ \mathrm{FPdim(\cH)} \mid \cH \text{ irreducible special Frobenius algebra in } \bimod(M) \right\}$.
\end{enumerate}
\end{remark}

We can prove Theorem \ref{thm:index-sets} now.

\begin{proof}[Proof of Theorem \ref{thm:index-sets}]
  Denote by $\cC$ the Haagerup fusion category \cite{asaedahaagerup99}. In \cite{grossmansnyder11}, possible principle graphs of irreducible special Frobenius algebras in $\cC$ are classified. Lemma 3.9 in \cite{grossmansnyder11} gives a list of possible principle graphs of non-trivial simple algebras in $\cC$. Note that the list of indices in Theorem \ref{thm:index-sets} is the same as the indices of graphs in \cite[Lemma 3.9]{grossmansnyder11}. We will refer with 1), 2), etc. to the graphs in this lemma. We prove that all the indices of these graphs, are actually realized by some irreducible special Frobenius algebra in $\cC$.

  Since, by \cite[Theorem 3.25]{grossmansnyder11}, there are three pairwise different categories that are Morita equivalent to $\cC$, all the possible principal graphs of minimal simple algebras are are actually realized by some irreducible special Frobenius algebra in $\cC$. So the graphs 1) and 3) are realized. Using the notation of \cite{grossmansnyder11} for irreducible objects in $\cC$, the graphs 4), 6) and 7) are realized by the irreducible special Frobenius algebras $\eta \overline{\eta}$, $\nu \overline{\nu}$ and $\mu \overline{\mu}$. We are left with the graphs 2) and 5). Theorem 3.25 in \cite{grossmansnyder11} gives the fusion rules for module categories over $\cC$. A short calculation shows that the square of the dimension of the second object in the module category associated with the Haagerup subfactor is the index of the graph 2). This proves that the graph 2) is realized. A similar calculation shows that the second object in the second non-trivial module category over $\cC$ gives rise to an irreducible special Frobenius algebra with principal graph given by 5).

  So all indices in \cite[Lemma 3.9]{grossmansnyder11} are actually attained by some irreducible special Frobenius algebra in $\cC$. According to Theorem \ref{thm:finite-tensor-categories} it is possible to find a II$_1$ factor $M$ such that $\bimod(M) \simeq \cC$. We conclude using Remark \ref{la:rem:characterisationindices}.
\end{proof}

%%% Local Variables: 
%%% mode: latex
%%% TeX-master: "construction-bimodule-categories-hyperfinite-case"
%%% End: 

%% file: appendix.tex
\section{Appendix}
\label{sec:appendix}

In this appendix, we prove that the category of finite dimensional unitary corepresentations of a discrete Kac algebra $A$, whose coopposite $A^{\mathrm{coop}}$ acts strictly outerly on the hyperfinite II$_1$ factor $R$, is realized as the the bimodule category of the inclusion $A \ltimes R \subset \wh{A} \ltimes A \ltimes R$. For convenience of the reader, we give a short introduction.

\subsection{Preliminaries on quantum groups}
\label{sec:discrete-quantum-groups}

\subsubsection{Locally compact quantum groups (see \cite{kustermansvaes00})}
\label{sec:locally-compact-quantum-groups}

A locally compact quantum group in the setting of von Neumann algebras is a von Neumann algebra $A$ equipped with a normal $*$-homomorphism $\Delta:A \ra A \vnt A$ and two normal, semi-finite, faithful  weights $\phi$, $\psi$ satisfying
\begin{itemize}
\item \textit{$\Delta$ is comultiplicative:} $(\id \otimes \Delta)\circ \Delta = (\Delta \otimes \id) \circ \Delta$.
\item \textit{$\phi$ is left invariant:} $\phi((\omega \otimes \id)(\Delta(x))) = \phi(x)\omega(1_M)$  for all $\omega \in M_*^+$ and all $x \in M^+$.
\item \textit{$\psi$ is right invariant:} $\psi((\id \otimes \omega)(\Delta(x))) = \psi(x)\omega(1_M)$  for all $\omega \in M_*^+$ and all $x \in M^+$.
\end{itemize}
We call $\Delta$ the comultiplication of $A$ and $\phi$, $\psi$ the left and the right Haar weight of $A$, respectively. If $\phi$ and $\psi$ are tracial, then $A$ is called a Kac algebra. If $A$ is of finite type I, then we say that $A$ is discrete. If $\phi$ and $\psi$ are finite, we say that $A$ is compact.

If $\Gamma$ is a discrete group, then $\linfty(\Gamma)$ is a discrete Kac algebra with comultiplication given by $\Delta(f)(g,h) = f(gh)$ and the left and right Haar weight both induced by the counting measure on $\Gamma$. 

For any locally compact quantum group $(A, \Delta)$ one can construct a dual locally compact quantum group $(\wh{A},\wh{\Delta})$ and a coopposite locally compact quantum group $A^{\mathrm{coop}}$. They both are represented on the same Hilbert space as $A$. Hence, it makes sense to write formulas involving elements of $A$ and $\wh{A}$ at the same time. We have $(A, \Delta) \cong (\wh{\wh{A}}, \wh{\wh{\Delta}})$ and $A$ is compact if and only if $\wh{A}$ is discrete.

\subsubsection{Corepresentations (see \cite{timmermann08})}
\label{sec:corepresentations}

A unitary corepresentation in $H$ of a locally compact quantum group $A$ is a unitary $U \in A \vnt \rB(H)$ such that $(\Delta \otimes \id)(U) = U_{23}U_{13}$. In what follows, we refer to unitary corepresentations simply as corepresentations. If $U \in \rB(H_U) \vnt A$ is a corepresentation of $A$, then we also refer to $U_{21} \in A \vnt \rB(H_U)$ as a corepresentation $A$.
A corepresentation $U$ of $A$ is called finite dimensional if $H_U$ is finite dimensional. The direct sum of two corepresentations $U, V$ of $A$ is denoted by $U \boxplus V \in A \vnt (\rB(H_U) \oplus \rB(H_V)) \cong A \vnt \rB(H_U) \oplus A \vnt \rB(H_V)$. The tensor product of two corepresentations $U$ and $V$ is given by $U \boxtimes V = U_{12}V_{13} \in A \vnt \rB(H_U) \vnt \rB(H_V)$. An intertwiner between two corepresentations $U$ and $V$ is a bounded linear map $T:H_U \ra H_V$ satisfying $(\id \otimes T)U = V(\id \otimes T)$. The space of all intertwiners between $U$ and $V$ is denoted by $\Hom(U,V)$. To every irreducible corepresentation $U \in A \vnt \rB(H_U)$ of $A$, one associates its conjugate corepresentation $(* \otimes \ol{\phantom{x}})(U) \in A \vnt \rB(\ol{H_U})$. Here $\ol{H_U}$ denotes the conjugate Hilbert space of $H_U$.  With this structure, the corepresentations of a locally compact quantum group $A$ become a tensor $\Cstar$-category $\CoRep(A)$. Its maximal compact tensor $\Cstar$-subcategory is the category of finite dimensional corepresentations $\CoRep_{\mathrm{fin}}(A)$. If $A$ is a compact quantum group, every irreducible corepresentation of $A$ is finite dimensional and every corepresentation is a direct sum of (possibly infinitely many) irreducible corepresentations. Coefficients of tensor products of arbitrary length of irreducible corepresentations of $A$ span it densely .

Let $A$ denote a compact quantum group. Then we can describe the evaluation of the Haar states on coefficients of corepresentations. In particular, $(\id \otimes \psi)(U) = (\id \otimes \phi)(U) = \delta_{U,\epsilon}  1$, where $\delta_{U,\epsilon}$ is $1$ if $U$ is the trivial corepresentation and $0$ otherwise.

If $A$ is discrete, its dual is compact. We can write $A$ as 
\[\bigoplus_{U \text{ irr. corep. of } \wh{A}}\rB(H_U) \eqstop\]
For any element $x \in A$ we can characterize $\Delta(x)$ as the unique element in $A \vnt A$ that satisfies $\Delta(x)T = Tx$ for all $T \in \Hom(U_1,U_2 \boxtimes U_3)$ and all irreducible corepresentations $U_1$,$U_2$ and $U_3$ of $\wh{A}$. Moreover, we can write any corepresentation $V \in A \vnt \rB(H_V)$ of $A$ as a direct sum of elements $V_U \in \rB(H_U) \otimes \rB(H_V)$ where $U$ runs through the irreducible corepresentations of $\wh{A}$. If $\epsilon$ denotes the trivial corepresentation, then $V_\epsilon = 1 \otimes 1$. Moreover, $V_{\ol{U}} = (\ol{\phantom{x}} \otimes *)(V)$. 

\subsubsection{Actions of quantum groups (see \cite{vaes05-strictly-outer-actions})}
\label{sec:actions-of-quantum-groups}

An action of a locally compact quantum group $A$ on a von Neumann algebra $N$ is a normal $*$-homomorphism $\alpha: N \ra A \vnt N$ such that $(\Delta \otimes \id) \circ \alpha = (\id \otimes \alpha) \circ \alpha$. The crossed product von Neumann algebra of $N$ by $\alpha$ is then the von Neumann algebra $A \ltimes N$  generated by $\wh{A} \otimes 1$ and $\alpha(N)$. We identify $N$ and $\wh{A}$ with subalgebras of $A \ltimes N$. There is a natural action $\wh{\alpha}$ of $\wh{A}$ on $A \ltimes N$, which is uniquely defined by $\wh{\alpha}(a) = \wh{\Delta}(a)$ for $a \in \wh{A}$ and $\wh{\alpha}(x) = 1 \otimes x$ for $x \in N$. This action is called the dual action of $\alpha$.

If an action $\alpha: N \ra A \vnt N$ of a locally compact quantum group on a factor satisfies $N' \cap A \ltimes N = \CC 1$, then $\alpha$ is called strictly outer.

Let $A$ be a discrete quantum group that acts via $\alpha$ on a von Neumann algebra $N$. We denote $A \ltimes N$ by $M$ and as before we identify $\wh{A}$ and $N$ with subalgebras of $M$. If $A$ is a Kac algebra, $N$ is finite and $\alpha$ preserves a trace $\tau_N$ on $N$, then $M$ is also finite. A faithful normal trace on $M$ is given by
\[(\tau \otimes \id)(U(1 \otimes x)) = \delta_{U,\epsilon} \tau_N(x) \eqcomma\]
for all $x \in N$ and for all irreducible corepresentations $U \in \wh{A} \otimes \rB(H_U)$ of $\wh{A}$. For $x \in N$ and $U \in \rB(H_U) \otimes \wh{A}$ an irreducible corepresentation of $\wh{A}$, we write $\alpha_U(x)$ for the projection of $\alpha(x)$ onto the direct summand $\rB(H_U) \otimes N$ of $A \vnt N$. For $x \in N$ we have $U(1 \otimes x)U^* = \alpha_U(x)$.

\subsection{Corepresentation categories of Kac algebras}
\label{sec:representation-categories-for-subfactors}

\begin{theorem}
\label{thm:representation-categories-for-subfactors}
  Let $N$ be a II$_1$ factor, $A$ a discrete quantum group and $\alpha: A \ra A \vnt N$ a strictly outer action. Denote by $M = A \ltimes N$ the crossed product of $N$ by $\alpha$ and write $\wh{A} \ltimes M$ for the crossed product by the dual action. Then $\bimod(M \subset \wh{A} \ltimes M) \simeq \CoRepfin(A^{\mathrm{coop}})$ as tensor $\Cstar$-categories, where $\CoRepfin(A^{\mathrm{coop}})$ denotes the category of finite dimensional corepresentations of $A^{\mathrm{coop}}$.
\end{theorem}
\begin{proof}
  We first construct a fully faithful tensor $\Cstar$-functor $F:\CoRepfin(A^{\mathrm{coop}}) \ra \bimod(M \subset \wh{A} \ltimes M)$. Then, we show that it is essentially surjective.

\textit{Step 1.}  Let $V \in A \otimes \Cmat{n}$ be a finite dimensional corepresentation of $A^{\mathrm{coop}}$, that is $(\Delta \otimes \id)(V) = V_{13}V_{23}$. We define a $*$-homomorphism $\psi:M \ra \Cmat{n} \otimes M$ such that
\[\psi(x) = 1 \otimes x \, \text{ for all }x \in N \quad \text{and} \quad (\id \otimes \psi)(U) = U_{13}V_{12} \eqcomma\]
where $U \in \rB(H_U) \otimes \wh{A} \subset A \vnt \wh{A}$ is an irreducible corepresentation of $\wh{A}$.

\textit{Proof of Step 1.} We first show that $\psi$ defines a $*$-homomorphism. This is obvious on $N$. In order to prove that $\psi$ is multiplicative on $\wh{A} \subset M$, we have to check for all irreducible corepresentations $U_1,U_2,U_3$ of $\wh{A}$ and for every intertwiner $T \in \Hom(U_1, U_2 \boxtimes U_3)$ the identity
\[(\id \otimes \psi)(U_2)_{134} (\id \otimes \psi)(U_3)_{234}(T \otimes \id)
  = (T \otimes \id)(\id \otimes \psi)(U_1)\]
holds. We have
\begin{align*}
  (T \otimes \id)(\id \otimes \psi)(U_1)
  & = 
  (T \otimes \id)U_{1,13}V_{U_1,12} \\
  & =
  U_{2,14}U_{3,24}(\Delta \otimes \id)(V_{U_1})_{123}(T \otimes \id) \\
  & =
  U_{2,14}U_{3,24}V_{U_2,13}V_{U_3,23}(T \otimes \id) \\
  & =
  U_{2,14}V_{U_2,13}U_{3,24}V_{U_3,23}(T \otimes \id) \\
  & =
  (\id \otimes \psi)(U_2)_{134} (\id \otimes \psi)(U_3)_{234}(T \otimes \id)
  \eqstop
\end{align*}
We prove that $\psi$ is a homomorphism on $\mathrm{alg}(\wh{A},N) = *\text{-}\mathrm{alg}(\wh{A},N)$. Using the fact that $U (1 \otimes x) = \alpha_U(x) U$ for all $x \in N$ and all irreducible corepresentations $U$ of $\wh{A}$, it suffices to note that
\begin{align*}
  (\id \otimes \psi)(U)(1 \otimes 1 \otimes x)
  & =
  U_{13}V_{12} (1 \otimes 1 \otimes x) \\
  & =
  U_{13}(1 \otimes 1 \otimes 1 \otimes x)V_{12} \\
  & =
  \alpha_U(x)_{13}U_{13}V_{12} \\
  & =
  \alpha_U(x)_{13}(\id \otimes \psi)(U)
  \eqstop
\end{align*}
Let us show that $\psi$ is $*$-preserving. We have
\begin{align*}
  (\id \otimes \psi)((\ol{\phantom{x}} \otimes *)(U)) =
  & =
  (\id \otimes \psi)(\ol{U}) \\
  & =
  \ol{U}_{13}V_{12} \\
  & = 
  \ol{U}_{13}V_{\ol{U},12} \\
  & =
  (\ol{\phantom{x}} \otimes *)(U)_{13}(\ol{\phantom{x}} \otimes *)(V_U)_{12} \\
  & =
  (\ol{\phantom{x}} \otimes * \otimes *)(U_{13}V_{12})
  \eqcomma
\end{align*}
This shows that $\psi$ is a $*$-homomorphism on $*\text{-}\mathrm{alg}(\wh{A},N)$.

Let us show that $\psi$ is trace preserving on $*\text{-}\mathrm{alg}(\wh{A}, N)$. Denote by $\tau$ the trace on $M$. For an irreducible corepresentation $U$ of $\wh{A}$ and $x \in N$ we have
\[(\id \otimes \tau)(U (1 \otimes x)) = \delta_{U,\epsilon} \tau(x) 1_A \eqcomma\]
by the definition of $\tau$. On the other hand we have
\begin{align*}
  (\id \otimes \tr \otimes \tau)(\id \otimes \psi)(U(1 \otimes x))
  & = 
  (\id \otimes \tr \otimes \tau)(U_{13}V_{12}(1 \otimes 1 \otimes x)) \\
  & =
  \delta_{U,\epsilon}\tau(x)(\id \otimes \tr)(V_{\epsilon, 12}) \\
  & =
  \delta_{U,\epsilon}\tau(x) 1_A
  \eqstop
\end{align*}
So $\psi$ is trace preserving and hence it extends to a $*$-homomorphism $\psi:M \ra \Cmat{n} \otimes M$. 

\textit{Step 2.}  Define a functor $F: \CoRepfin(A^{\mathrm{coop}}) \ra \bimod(Q \subset Q_1)$ such that if $V$ is a  finite dimensional corepresentation of $A^{\mathrm{coop}}$ and $\psi$ the map associated with it in Step 1, we have $F(V) = \cH(\psi)$. If $T \in \Hom(V_1,V_2)$ is an intertwiner, we set $F(T) = T \otimes \id: H_{V_1} \otimes \Ltwo(M) \ra H_{V_2} \otimes \Ltwo(M)$. Then $F$ is fully faithful tensor \Cstar-functor.

\textit{Proof of Step 2.} It is obvious that $F$ is faithful. In order to show that $F$ is full, let $V_1 \in A \otimes \Cmat{m}$,$V_2 \in A \otimes \Cmat{n}$ be finite dimensional corepresentations of $A^{\mathrm{coop}}$. Denote by $\psi_1$, $\psi_2$ the maps associated with $V_1$ and $V_2$, respectively. Let $T: \CC^m \otimes \Ltwo(M) \ra \CC^n \otimes \Ltwo(M)$ be an intertwiner from $\cH(F(V_1))$ to $\cH(F(V_2))$. Then $T \in \rB(\CC^m, \CC^n) \otimes M$ satisfies
\[ T (1 \otimes x) =  T \psi_1(x) = \psi_2(x) T = (1 \otimes x) T \quad \text{for all }x \in N \eqstop\]
Hence, $T \in \rB(\CC^m, \CC^n) \otimes 1$. So, for any irreducible corepresentation $U$ of $\wh{A}$, we have
\begin{align*}
  V_{2,12}T_{23}
  & =
  U_{13}^*U_{13}V_{2,12}T_{23} \\
  & =
  U_{13}^*\psi_2(U)T_{23} \\
  & =
  U_{13}^*T_{23}\psi_1(U) \\
  & =
  T_{23}U_{13}^*U_{13}V_{1,12} \\
  & =
  T_{23}V_{1,12}
  \eqstop
\end{align*}
So $T$ comes from an intertwiner from $V_1$ to $V_2$. This shows that $F$ is full.

For an intertwiner $T \in \Hom(V_1, V_2)$ we have $F(T^*) = F(T)^*$, so $F$ is a $\Cstar$-functor.

If $V_1$,$V_2$ are finite dimensional corepresentations of $A^{\mathrm{coop}}$, $\psi_1, \psi_2$ and $\psi$ denote the maps associated with $V_1, V_2$ and $V_{1,12}V_{2,13} = V_1 \boxtimes V_2$ respectively, then $(\id \otimes \psi)(U) = U_{14}V_{1,12}V_{2,13} = (\id \otimes \psi_2) \circ \psi_1(U)$, for every irreducible corepresentation $U \in \rB(H_U) \otimes \wh{A}$ of $\wh{A}$. So $\psi = (\id \otimes \psi_2) \circ \psi_1$. We obtain $F(V_1) \otimes_M F(V_2) \cong F(V_1 \boxtimes V_2)$ and this unitary isomorphism is natural in $V_1$ and $V_2$. Hence $F$ is a tensor $\Cstar$-functor.  

\textit{Step 3.}  $F$ is essentially surjective.

\textit{Proof of Step 3.} Let $\cH$ be a finite index bimodule in $\bimod(M \subset \wh{A} \ltimes M)$. Write $\cH \cong \cH(\psi)$ for some $\psi: M \ra p(\Cmat{n} \otimes M)p$ satisfying $p \in (1 \otimes N)' \cap (\Cmat{n} \otimes M)$ and $\psi(x) = p(1 \otimes x)$ for all $x \in N$. Since $N \subset M$ is irreducible, we have $p \in \Cmat{n} \otimes 1$, so we may assume that $p = 1$. For an irreducible corepresentation $U$ of $\wh{A}$, by the same calculation as in Step $1$, we obtain
\[(\id \otimes \psi)(U)U_{13}^*\alpha_{U}(x) = \alpha_{U}(x)(\id \otimes \psi)(U)U_{13}^* \eqcomma\]
for all $x \in N$. Since $N$ is linearly generated by the coefficients of $\alpha_U(N)$, it follows that
\[(\id \otimes \psi)(U)U_{13}^* = V_{U,12}\]
for some element in $V_U \in \rB(H_U) \otimes \Cmat{n} \subset A \otimes \Cmat{n}$. Let
\[V = \bigboxplus_{U \text{ irr. corep. of } \wh{A}} V_U \in A \otimes \Cmat{n} \eqstop\]
We show that $V$ is a corepresentation of $A^{\mathrm{coop}}$, i.e. that $(\Delta \otimes \id)(V) = V_{13}V_{23}$. It suffices to show for any irreducible corepresentations $U_1, U_2, U_3$ and any intertwiner $T \in \Hom(U_1, U_2 \boxtimes U_3)$ that we have
\[V_{U_2,13}V_{U_3,23} (T \otimes \id) = (T \otimes \id)(V_{U_1}) \eqstop\]
Indeed, we have
\begin{align*}
  (T \otimes \id)(V_{U_1} \otimes 1)
  & =
  (T \otimes \id)(\id \otimes \psi)(U_1)U_{1,13}^* \\
  & =
  (\id \otimes \psi)(U_2)_{134}(\id \otimes \psi)(U_3)_{234}(U_{2,14}U_{3,24})^*(T \otimes \id) \\
  & =
  (\id \otimes \psi)(U_2)_{134}V_{3,23}U_{2,14}^*(T \otimes \id) \\
  & =
  (\id \otimes \psi)(U_2)_{134}U_{2,14}^*V_{3,23}(T \otimes \id) \\
  & =
  V_{2,13}V_{3,23}(T \otimes \id)
  \eqstop
\end{align*}
This shows that $V$ is a corepresentation of $A$ and $\cH(\psi) = F(V)$.
\end{proof}

%%% Local Variables: 
%%% mode: latex
%%% TeX-master: "construction-bimodule-categories-hyperfinite-case"
%%% End: 